\documentclass[a4paper,11pt]{article}
\usepackage{makeidx}
\usepackage[mathscr]{eucal}
\usepackage[dvips]{color}
\usepackage{amssymb, amsfonts, amsmath, amsthm, graphics} 
\usepackage{amsbsy}
\usepackage{mathrsfs}
\usepackage[colorlinks=true]{hyperref}
\usepackage{here}

\newtheorem{definition}{Definition}[section]
\newtheorem{proposition}{Proposition}[section]
\newtheorem{theorem}{Theorem}[section]

\newtheorem{assumption}{Assumption}[section]
\newtheorem{lemma}[proposition]{Lemma}
\newtheorem{remark}{Remark}[section]

\newtheorem*{thank}{Acknowledgments}

\numberwithin{equation}{section}

\newcommand{\N}{\mathbb{N}}

\newcommand{\R}{\mathbb{R}}

\newcommand{\T}{\mathbb{T}}

\newcommand{\boF}{\mathcal{F}}

\newcommand{\boQ}{\mathcal{Q}}

\begin{document}
\title{Remarks on solitary waves and Cauchy problem for a Half-wave-Schr\"{o}dinger equations}

\author{\renewcommand{\thefootnote}{\arabic{footnote}} Yakine Bahri\footnotemark[1], \ Slim Ibrahim\footnotemark[2] \ and Hiroaki Kikuchi\footnotemark[3]}
\footnotetext[1]{Department of Mathematics and Statistics, University of Victoria 3800 Finnerty Road, Victoria, B.C., Canada V8P 5C2 . E-mail: {\tt ybahri@uvic.ca}}

\footnotetext[2]{Department of Mathematics and Statistics, University of Victoria 3800 Finnerty Road, Victoria, B.C., Canada V8P 5C2 . E-mail: {\tt ibrhims@uvic.ca}}

\footnotetext[3]{Department of Mathematics, Tsuda University 2-1-1 Tsuda-machi, Kodaira-shi, Tokyo 187-8577, JAPAN. E-mail: {\tt hiroaki@tsuda.ac.jp}}
\date{}
\pagestyle{plain}
\maketitle

\begin{abstract}
In this paper, we study the solitary wave and 
the Cauchy problem for Half-wave-Schr\"{o}dinger equations in the plane. First, we show the existence and orbital stability of the ground states. Secondly, we prove that traveling waves exist and converge to zero as the velocity tends to $1$. Finally, we solve the Cauchy problem for initial data in $L^{2}_{x}H^{s}_{y}(\mathbb{R}^{2})$, with $s>\frac{1}{2}$. 
\end{abstract}
\section{Introduction}
In this paper, we consider the  
the following Half-wave-Schr\"{o}dinger equation: 
\begin{equation} \label{WS}
i \partial_{t} \psi + \partial_{xx} \psi - |D_{y}| \psi 
+ |\psi|^{p-1} \psi = 0, \qquad 
\mbox{in $\mathbb{R} \times \mathbb{R}^{2}$}, 
\end{equation}
where 
$i = \sqrt{-1}, 1 < p<5$ and 
$|D_{y}| := \sqrt{- \partial_{yy}}$. 

The equation \eqref{WS} was first considered by Xu~\cite{Xu}. 
She studied the large time behavior of solutions to the cubic defocusing Half-wave-Schr\"{o}dinger equation on spatially cylinder $\mathbb{R}_{x} \times \mathbb{T}_{y}$ with small smooth initial data and obtained modified scattering results. However, it seems that 
little is known for large initial data. In this paper, we shall consider the focusing case because it provides particular solutions of type traveling and standing waves. We first study these two kinds of solitary waves to \eqref{WS}.

Let $X := H^{1}_{x} L^{2}_{y}(\mathbb{R}^{2}) 
\cap L^{2}_{x}H^{\frac{1}{2}}_{y}(\mathbb{R}^{2})$ with the norm 
\[
\|u\|_{X} := \left\{\|\partial_{x} u\|_{L^{2}}^{2} 
+ \||D_{y}|^{\frac{1}{2}} u\|_{L^{2}}^{2} + \|u\|_{L^{2}}^{2} 
\right\}^{\frac{1}{2}}. 
\] 
Concerning the stability of the standing waves, we always assume 
that 
\eqref{WS} is locally well-posed in X.
Namely, we assume the following:
\begin{assumption} \label{assum-1}
For any $\psi_{0} \in X$, there exist $T = T(\|\psi_{0}\|_{X}) >0$ and 
a unique local solution $u \in C([0, T), X)$ to \eqref{WS} 
with $\psi|_{t=0} = \psi_{0}$. 
Moreover, as long as the solution exists, the following conservation laws holds; 
\begin{equation} \label{conserv}
\mathcal{H}(\psi(t)) = \mathcal{H}(\psi_{0}), \qquad 
\mathcal{M}(\psi(t)) = \mathcal{M}(\psi_{0}) \qquad
\mbox{for all $t \in I_{\max}$},  
\end{equation}
where $I_{\max}$ is the maximal existence time and 
\begin{align*}
& \mathcal{M}(u) 
 := \frac{1}{2} \|u\|_{L^{2}(\mathbb{R}^{2})}^{2}, \\
& \mathcal{H}(u) 
 := \frac{1}{2} \iint_{\mathbb{R}^{2}}
\left(|\partial_{x} u(x, y)|^{2} + 
|D_{y}| u(x, y) \overline{u(x, y)}\right) dx 
 - \frac{1}{p+1}\iint_{\mathbb{R}^{2}} |u(x, y)|^{p+1} dxdy. 
\end{align*}  
\end{assumption}
\begin{remark}
As pointed out by Xu~\cite{Xu}, the Cauchy problem 
of Schr\"{o}dinger-half-wave equations \eqref{WS} 
is not easy because usual technics are not helpful since its Hamiltonian energy lies on the Sobolev spaces $X = H^{1}_{x}L^{2}_{y} \cap L_{x}^{2}H_{y}^{\frac{1}{2}}(\mathbb{R}^{2})$ (see Remark \ref{Rk1.6} for more details). However, we can construct a local solution in 
$L_{x}^{2}H_{y}^{\frac{1}{2} + \varepsilon}(\mathbb{R}^{2})$ for any $\varepsilon >0$ (see Section \ref{section-Cauchy-Problem} below). Thus, we may expect that the equation \eqref{WS} is locally-well posed in the energy space $X$. 
\end{remark}
The equation \eqref{WS} is scale-invariant.  Namely, if $\psi(t, x, y)$ is a solution to \eqref{WS}, 
\[
\psi_{\lambda}(t, x, y) = \lambda^{\frac{2}{p-1}}\psi(\lambda^{2}t, 
\lambda x, \lambda^{2}y)
\]
also satisfies \eqref{WS} 
for all $\lambda>0$. 
We note that 
if we put 
\[
s_{p} := \frac{3}{2} - \frac{2}{p-1}, 
\]
then we have 
\[
\|\psi_{\lambda}|_{t=0}\|_{\dot{H}_{x}^{s_{p}}L_{y}^{2}} = 
\|\psi_{0}\|_{\dot{H}^{s_{p}}L_{y}^{2}}, \qquad 
\|\psi_{\lambda}|_{t=0}\|_{L_{x}^{2}\dot{H}_{y}^{ \frac{s_{p}}{2}}  } = 
\|\psi_{0}\|_{L_{x}^{2}\dot{H}_{y}^{\frac{s_{p}}{2} } } \qquad 
\mbox{for all $\lambda > 0$.}
\]
Thus, the case $p=7/3$ is the so-called $L^{2}$-critical. 
\par
We first study the standing waves to \eqref{WS}. 
By a standing wave, we mean a solution to \eqref{WS} 
of the form $\psi(t, x, y) = e^{i \omega t}Q_{\omega}(x, y) 
\; (\omega >0)$. 
Then, $Q_{\omega}$ satisfies the following elliptic equation:
\begin{equation} \label{sp}
- \partial_{xx} Q_{\omega} + |D_{y}|Q_{\omega} + \omega Q_{\omega} - |Q_{\omega}|^{p-1}Q_{\omega} = 0 \quad \mbox{in} \  \mathbb{R}^{2}. 
\end{equation}
The equation \eqref{sp} also appears in the 
stationary problem 
to the Benjamin-Ono-Zakharov-Kuznetsov equation
(see e.g. \cite{Esfahani-Pastor1, Esfahani-Pastor2}). 
We note that if we put 
\[
\mathcal{S}_{\omega}(u) := 
\frac{1}{2}\|\partial_{x}u\|_{L^{2}}^{2} + 
\frac{1}{2}\||D_{y}|^{\frac{1}{2}}u\|_{L^{2}}^{2} 
+ \frac{\omega}{2}\|u\|_{L^{2}}^{2}
- \frac{1}{p+1}\|u\|_{L^{p+1}}^{p+1},  
\] 
then we see that $Q_{\omega} \in X$ is a solution 
to \eqref{sp} if and only if $Q_{\omega}$ is a critical point of 
the functional $\mathcal{S}_{\omega}$. 
We pay our attention to the ground states. 
The ground state is a solution to \eqref{sp} which minimizes 
the corresponding functional $\mathcal{S}_{\omega}$ for 
all non-trivial solution to \eqref{sp}. 
Concerning the existence of the ground states, 
we obtain the following theorem: 
\begin{theorem} \label{existence-groundstate}
Let $1 < p < 5$ and $\omega >0$. 
Then there exist ground states $Q_{\omega} \neq 0$ which 
minimize the following constrained minimization 
problem:   
\[
m_{\omega} = \inf \left\{
\mathcal{S}_{\omega} (u) \mid u \in X \setminus \{0\}, \; 
\mathcal{N}_{\omega} (u) = 0
\right\},  
\]
where 
\[
\mathcal{N}_{\omega}(u) 
:= \|\partial_{x}u\|_{L^{2}}^{2} + 
\||D_{y}|^{\frac{1}{2}}u\|_{L^{2}}^{2} 
+ \omega \|u\|_{L^{2}}^{2}
- \|u\|_{L^{p+1}}^{p+1}. 
\]
\end{theorem}
\begin{remark}
Since the equation \eqref{sp} is anisotropic, we cannot 
expect that the ground state to \eqref{sp} is radially symmetric 
(see Esfahani, Pastor and Bona~\cite{Esfahani-Pastor2} 
for the symmetry of the ground state)
\end{remark}
Theorem \ref{existence-groundstate} can be obtained as a corollary of 
Theorem \ref{travel-thm1} below. 
We will explain it later. 
 
Next, we are concerned with the stability of standing waves. 
The stability is defined by the following: 
\begin{definition}
\begin{enumerate}
Assume that Assumption \ref{assum-1} holds. 
\item[\rm (i)] 
Let $\Sigma \subset X$. 
We say that $\Sigma$ is {\it stable} if for any 
$\varepsilon >0$, there exists $\delta >0$ such that 
if $\psi_{0} \in X$ satisfies $\inf_{u \in \Sigma} \|\psi_{0} - u\|_{X} < \delta$, 
then the solution $\psi(t)$ with $\psi|_{t=0} = \psi_{0}$ satisfies 
\[
\sup_{t \in \mathbb{R}} \inf_{u \in \Sigma} \|\psi(t) - u\|_{X} < \varepsilon. 
\]
Otherwise, we say that $\Sigma$ is {\it unstable}. 
\item[\rm (ii)]
We say that $e^{i \omega t}Q_{\omega}$ is {\it orbitally stable} if its orbit 
\begin{equation} \label{orbit}
\mathcal{O}_{\omega} = \left\{ 
e^{i \theta} Q_{\omega}(\cdot + \tau_{1}, \cdot + \tau_{2}) \mid 
\theta \in \mathbb{R}, (\tau_{1}, \tau_{2}) \in \mathbb{R}^{2}
\right\}
\end{equation}
is stable. 
\end{enumerate}
\end{definition}
We first consider the $L^{2}$-sub-critical case $1 < p< 7/3$. 
In this case, 
we shall show that some subset, which includes the orbit of ground states, 
is stable by using the method of Cazenave and P.-L. Lions~\cite{Cazenave-Lions}. 
To state our result, we prepare several notations. 
For each $\mu > 0$, we consider the following minimization problem: 
\[
I(\mu) := \inf \left\{\mathcal{H}(u) \mid 
\mathcal{M}(u) = \mu \right\}, 
\]
and we denote by $\Sigma(\mu)$ the set of minimizers, that is, 
\[
\Sigma(\mu) := \left\{u \in X \mid \mathcal{H}(u) = I(\mu), 
\; \mathcal{M}(u) = \mu \right\}. 
\]
Then, we obtain the following stability result:  
\begin{theorem} \label{L2subl-thm2-0}
Assume that Assumption \ref{assum-1} holds. 
Let $1 < p < \frac{7}{3}$. 
For each $\mu >0$, the set $\Sigma(\mu)$ is 
non-empty and stable under the flow of \eqref{WS}. 
\end{theorem}
\begin{remark}
\begin{enumerate}
\item[\rm (i)]
Let $Q_{\omega}$ be the ground state to \eqref{sp}. 
Then, we have 
\begin{equation} \label{scale-Q}
Q_{\omega}(x, y) = \omega^{\frac{1}{p-1}} Q_{1}(\sqrt{\omega}x, \omega y). 
\end{equation}
Thus, putting 
\[
\omega(\mu) = \left(\frac{2\mu}{\|Q_{1}\|_{L^{2}}^{2}} \right)^{-\frac{1}{s_{p}}} 
\]
for each $\mu > 0$, 
we see that $Q_{\omega(\mu)} \in \Sigma(\mu)$. 
\item[{\rm (ii)}]
From {\rm (i)}, we see that $\mathcal{O}_{\omega(\mu)} 
\subset \Sigma(\mu)$ for each 
$\mu >0$. 
Thus, in order to obtain the stability of standing wave $e^{i \omega t}Q_{\omega}$ 
from the result of Theorem \ref{L2subl-thm2-0}, 
it is enough to show the uniqueness of the ground state. 
However, it is challenging problem to show the uniqueness. 
Recently, uniqueness of the ground states for the half-wave equations 
was proved in \cite{Frank-Lenzmann1, Frank-Lenzmann2}. 
 However, at least for the authors, 
it is not clear whether we can apply their method to 
the equation \eqref{sp} because of the anisotropy.  
\end{enumerate}
\end{remark}
Next, we state our instability result.
\begin{theorem} \label{instability}
Assume that Assumption \ref{assum-1} holds. 
Let $\frac{7}{3} < p < 5$ and $Q_{\omega}$ 
be the ground state to \eqref{sp}. 
For any $\omega >0$, the standing wave $e^{i \omega t}Q_{\omega}$ 
is unstable. 
\end{theorem}
Grillakis, Shatah and Strauss~\cite{Grillakis-Shatah-Strauss} gave 
a sufficient condition for stability of solitary waves for abstract 
Hamiltonian systems. 
To apply the result of \cite{Grillakis-Shatah-Strauss}, 
we need to check the non-degeneracy of the ground states. 
As well as the uniqueness, 
to show the non-degeneracy is also a challenging problem. 
To avoid the difficulty, we employ the argument 
of \cite{Shatah-Strauss} and \cite{Concalves}, 
which exploits the variational characterization of the ground state 
instead of the spectral properties of the 
linearized operator.  

Finally, we shall study a existence of 
traveling wave solution. 
By a traveling wave, we mean a solution to \eqref{WS} 
of the form
\[
u(t, x, y) = e^{i \omega t} Q_{\omega, v}(x, y - vt), 
\]
where $\omega >0$ and $v \in \mathbb{R}$. 
Then, we see that $Q_{\omega, v}$ satisfies 
\begin{equation} \label{spv}
- \partial_{xx} Q + |D_{y}| Q - i v \partial_{y}Q + \omega Q 
- |Q|^{p-1}Q = 0, \qquad \mbox{in $\mathbb{R}^{2}$}. 
\end{equation}
If we put 
\begin{equation*}
\begin{split}
\mathcal{S}_{\omega, v}(u)  = \frac{1}{2} \int_{\mathbb{R}^{2}} \Big\{ & |\partial_{x} u(x, y)|^{2} + |D_{y}| u(x, y) \overline{u(x, y)}  
- i v \partial_{y} u(x, y) \overline{u(x, y)} \\
&  + \omega |u(x, y)|^{2} \Big\} dxdy  - \frac{1}{p+1} \int_{\mathbb{R}^{2}} |u(x, y)|^{p+1} dxdy,  
\end{split}
\end{equation*}
then, $\mathcal{S}_{\omega, v}$ is well-defined on $X$ and 
$u \in X$ is a solution to \eqref{spv} if and only if $u \in X$ is a
critical point of $\mathcal{S}_{\omega, v}$. 
We seek a critical point of $\mathcal{S}_{\omega, v}$ by considering 
the following minimization problem: 
\begin{equation}
\label{def:m-omega-v}
m_{\omega, v} := \inf\left\{ 
\mathcal{S}_{\omega, v}(u) \mid u \in X \setminus \{0\}, \; 
\mathcal{N}_{\omega, v}(u) = 0
\right\}, 
\end{equation}

where 
\begin{equation}
\label{def:N-omega-v}
\begin{split}
\mathcal{N}_{\omega, v}(u) 
& = \langle \mathcal{S}_{\omega, v}^{\prime}(u), v\rangle \\
& = \int_{\mathbb{R}^{2}} \left\{ 
|\partial_{x} u(x, y)|^{2} + |D_{y}| u(x, y) \overline{u(x, y)}  
- i v \partial_{y} u(x, y) \overline{u(x, y)} + \omega |u(x, y)|^{2} \right\} dxdy \\
& \quad - \int_{\mathbb{R}^{2}} |u(x, y)|^{p+1} dxdy,  
\end{split}
\end{equation}
We first show the following: 
\begin{theorem} \label{travel-thm1}
Let $\omega >0, v \in \mathbb{R}$ with $|v| < 1$ and $1 < p< 5$. 
Then, there exists a non-trivial solution $Q_{\omega, v} \in X$ 
to \eqref{spv} which satisfies 
$\mathcal{S}_{\omega, v}(Q_{\omega, v}) = m_{\omega, v}$. 
\end{theorem}
\begin{remark}
Note that in case $v = 0$, the equations coincides with 
\eqref{sp}. 
Therefore, 
Theorem \ref{existence-groundstate} can be obtained as a corollary 
of Theorem \ref{travel-thm1}. 
\end{remark}
Secondly, we shall study the convergence of the solution $Q_{\omega, v}$, 
which is obtained in Theorem \ref{travel-thm1}, as $|v|$ goes to $1$.  
Concerning with this, we obtain the following: 
\begin{theorem} \label{travel-thm12}
Let $\omega > 0, v \in \mathbb{R}$ with $|v| < 1$, $1 < p< 5$. 
and $Q_{\omega, v} \in X$ be the solution to \eqref{spv} 
satisfying $\mathcal{S}_{\omega, v}(Q_{\omega, v}) = m_{\omega, v}$. 
Then, we have 
\[
\lim_{|v| \to 1} \|\partial_{x}Q_{\omega, v}\|_{L^{2}} = 
\lim_{|v| \to 1} \|Q_{\omega, v}\|_{L^{2}} = 0. 
\]
\end{theorem}
\begin{remark}
\begin{enumerate}
\item[\rm (i)]
We should remark that  
similar results of Theorem \ref{travel-thm12} 
were already observed for the cubic half-wave equation (see \cite{Bellazzini-Georgiev-Visciglia} and \cite{Krieger-Lenzmann-Raphael}). 
\item[\rm (ii)] As an extension to \cite{Krieger-Lenzmann-Raphael}, G\'{e}rard, Lenzmann, Pocovnicu and Rapha\"{e}l studied in \cite{G-L-P-R} the behavior of traveling waves as the velocity $v$ goes to $\pm1$. Using an appropriate scaling, they obtained a minimizing traveling wave of the cubic Szeg\"{o} equation which provides the uniqueness of traveling waves for the cubic Half-wave equation. 
\item[\rm (iii)]
Dodson~\cite{Dodson} considered the following focusing  
mass critical nonlinear  Schr\"{o}dinger equations:
\begin{equation} \label{mass-NLS}
i \frac{\partial \psi}{\partial t} + \Delta \psi + 
|\psi|^{\frac{4}{d}} \psi = 0 \qquad 
\mbox{in $\mathbb{R} \times \mathbb{R}^{d}$}. 
\end{equation} 
Then, it was shown in \cite{Dodson} that 
any solution to \eqref{mass-NLS} whose $L^{2}$-norm 
is below that of the ground state scatters. 
Contrary to the equation \eqref{mass-NLS}, 
from the result of Theorem \ref{travel-thm12}, 
we see that there exists a non-dispersive solution to \eqref{WS}  
whose $L^{2}$ norm is below that of the ground state. 
\end{enumerate}
\end{remark}
Finally, we will study the Cauchy problem of \eqref{WS} and show the following theorem:
\begin{theorem} \label{Cauchy}
Assume that $1 < p \leq 5$ and $s > 1/2$. 
For any $\psi_{0} \in L_{x}^{2}H_{y}^{s}(\mathbb{R}^{2})$, 
there exist $T_{\max} > 0$ and a unique local solution $u \in C((-T_{\max}, T_{\max}); L_{x}^{2}H_{y}^{s}(\mathbb{R}^{2}))$ 
to \eqref{WS} with $\psi|_{t = 0} = \psi_{0}$.  
\end{theorem}
\begin{remark}
\label{Rk1.6}
\begin{itemize}
\item[{\rm (i)}] Here, we will show that maximal time of existence $T_{\max} $ depends only on $\|\psi_{0}\|_{L_{x}^{2}H_{y}^{s}(\mathbb{R}^{2})}$ when $1<p<5$. 
We do not know whether this is still true for $p=5$ (see the proof of Remark \ref{Rk5.3} for more details).
\item[{\rm (ii)}] In the cubic defocusing case, Xu has shown in \cite{Xu}, as a direct consequence of her result, the global existence of solutions with small initial datum in $L^{2}_x H^{s}_y(\R \times \T )$ with $s>\frac{1}{2}.$  
\item[{\rm (iii)}] Unfortunately we are unable to solve the Cauchy problem in the energy spave $X$. The difficulty comes from the total absence of dispersion in $y$ variable. Thus, we cannot hope to loose regularity less then the Sobolev embedding in the Strichartz estimate. In addition, the argument used for the $1$D Half-wave equation (see Appendix D in \cite{Krieger-Lenzmann-Raphael} for more details) cannot be applied in this setting due to the embedding of the energy space $X$ in $L^q_{x,y}(\R^2)$ for only $q\leq6$. Also, Brezis-Gallouet type estimate is not useful due to the anisotropic property of the equation (see for example \cite{Ib-Maj-Mas}).
\end{itemize}
\end{remark}  
The proof relies on the 
contraction mapping argument using the Duhamel formula and a suitable Strichartz estimate. For that, we will obtain the Strichartz estimate using Keel-Tao Lemma in \cite{Keel-Tao} and frequency decomposition in $y$ variable.   

This paper is organized as follows. 
In Section \ref{section-stability}, we show the stability 
of the set of minimizer. 
In Section \ref{section-instability}, we give a proof of the instability 
result. 
In Section \ref{section-travelingwave}, we study the traveling waves to 
\eqref{WS}. 
In Section \ref{section-Cauchy-Problem}, we prove the 
local well-posedness for the equation \eqref{WS} with initial datum  slightly regular than the energy space in $y$ variable. In the appendix, we recall some regularity results for the ground state and useful tools for compactness.

We give the notations used in this paper: 
\begin{enumerate}
\item[\rm (i)]
For any $\Omega \subset \mathbb{R}^{2}$,  
we use $(\cdot, \cdot)_{L^{2}_{x,y}(\Omega)}$ 
to denote the inner product of $L^{2}(\Omega)$: 
\[
(f, g)_{L^{2}_{x,y}(\Omega)} := 
\int_{\Omega} f(x, y) \overline{g(x, y)} dxdy. 
\]
\item[\rm (ii)]
For any $\Omega \subset \mathbb{R}^{2}$,  
$X(\Omega)$ denotes the Hilbert space equipped with the following 
inner product: 
\[
(u, v)_{X(\Omega)} = (\partial_{x}u, \partial_{x}v)_{L^{2}_{x,y}(\Omega)} 
+ (|D_{y}|u, |D_{y}|v)_{L^{2}_{x,y}(\Omega)} 
+ (u, v)_{L^{2}_{x,y}(\Omega)}.   
\]
When $\Omega = \mathbb{R}^{2}$, 
we may write $L^{2}$ and $X$ in short 
if there is no risk of confusion. 
\item[\rm (iii)]
We use $U_{\varepsilon}(Q_{\omega})$ to denote a tubular 
neighborhood around the orbit 
$\{e^{i s}Q_{\omega} \mid s \in \mathbb{R}\}$, namely, 
\[
U_{\varepsilon}(Q_{\omega}) := \left\{
u \in X \mid \inf_{\theta \in \mathbb{R}} 
\|u - e^{i \theta}Q_{\omega}\|_{X} < \varepsilon
\right\}
\]
\item[\rm (iii)] 
For the space-time Lebesgue norms, we use the following notations:
$$L^r_tL^{q_1}_xL^{q_2}_y := L^r (\R, L^{q_1}_xL^{q_2}_y(\R^2))$$
and
$$L^r_TL^{q_1}_xL^{q_2}_y := L^r ((-T,T), L^{q_1}_xL^{q_2}_y(\R^2)),$$
with $T>0$ and $r,$ $q_1$ and $q_2$ are in $[1, \infty].$ 
\end{enumerate}
\section{Stability result}\label{section-stability}
In this section, we shall show Theorem \ref{L2subl-thm2-0} 
following Cazenave and P.-L. Lions~\cite{Cazenave-Lions}. 
To this end, we prepare several lemmas. 
We first recall the following 
anisotropic Gagliardo-Nirenberg inequality: 
\begin{lemma}[Anisotropic Gagliardo-Nirenberg inequality]
Assume that $1 < p< 5$. 
There exists a constant $C_{GN} >0$ such that 
\begin{equation} \label{GN-ineq}
\|u\|_{L^{p+1}}^{p+1} \leq C_{GN} 
\|\partial_{x} u\|_{L^{2}}^{\frac{p-1}{2}} 
\||D_{y}|^{\frac{1}{2}} u\|_{L^{2}}^{p-1}
\|u\|_{L^{2}}^{\frac{5-p}{2}}
\end{equation}
for all $u \in X$. 
\end{lemma}
We note that the best constant is attained by the ground sate 
to \eqref{sp} (see \cite{Esfahani-Pastor1} in detail). 
\begin{lemma} \label{L2subl-thm2-1}
For each $\mu >0$. we have $-\infty < I(\mu) < 0$. 
\end{lemma}
\begin{proof}
First, we shall show that $I(\mu) < 0$. 
Let $u \in X$ with $\mathcal{M}(u) = \mu$. 
We consider the following $L^{2}$-scaling:
\begin{equation} \label{L2scale}
T_{\lambda}(\cdot, \cdot) = \lambda^{\frac{3}{4}} u(\lambda^{\frac{1}{2}} \cdot, 
\lambda \cdot)\qquad (\lambda >0). 
\end{equation}
Then, we see that $\|T_{\lambda} u\|_{L^{2}} = \|u\|_{L^{2}}$ and 
\[
\mathcal{H}(T_{\lambda} u) 
= \frac{\lambda}{2} \left(\|\partial_{x} u\|_{L^{2}}^{2} 
+ \||D_{y}|^{\frac{1}{2}} u\|_{L^{2}}^{2} \right)  
- \frac{\lambda^{\frac{3}{4}(p-1)}}{p+1} \|u\|_{L^{p+1}}^{p+1}. 
\]
Note that $\frac{3}{4}(p-1) < 1$ for $1 < p< \frac{7}{3}$. 
Therefore, we can take $\lambda >0$ sufficiently small so that 
$\mathcal{H}(T_{\lambda} u) < 0$. 
This yields that $I(\mu) < 0$. 

Next, we prove that $I(\mu) > -\infty$. 
By the Gagliardo-Nirenberg inequality \eqref{GN-ineq} 
and $\mathcal{M}(u) = \mu$, 
we have 
\begin{equation*}
\begin{split}
\mathcal{H}(u)
& = \frac{1}{2} \left(\|\partial_{x} u\|_{L^{2}}^{2} 
+ \||D_{y}|^{\frac{1}{2}} u\|_{L^{2}}^{2} \right) \\
& \quad - 
\frac{C_{GN}}{p+1} 
\||D_{y}|^{\frac{1}{2}} u\|_{L^{2}}^{p-1} 
\|\partial_{x} u\|_{L^{2}}^{\frac{p-1}{2}} 
\|u\|_{L^{2}}^{\frac{5-p}{2}}\\
& \geq \frac{1}{2} 
\left(\|\partial_{x} u\|_{L^{2}}^{2} 
+ \||D_{y}|^{\frac{1}{2}} u\|_{L^{2}}^{2} \right)
 - 
\frac{C_{GN}}{p+1} \mathcal{M}(u)^{\frac{5-p}{4}} \left(\|\partial_{x} u\|_{L^{2}}^{2} 
+ \||D_{y}|^{\frac{1}{2}} u\|_{L^{2}}^{2} \right)^{\frac{3(p-1)}{2}}. 
\end{split}
\end{equation*}
We remark that $\frac{3(p-1)}{2} < 1$ for $1 < p< \frac{7}{3}$. 
Thus, by the Young inequality, we have 
\begin{equation} \label{eq2-1}
\mathcal{H}(u) \geq \frac{1}{4} 
\left(\|\partial_{x} u\|_{L^{2}}^{2} 
+ \||D_{y}|^{\frac{1}{2}} u\|_{L^{2}}^{2} \right) - C 
\geq - C, 
\end{equation}
where $C>0$ is some constant, which is independent on $u \in X$. 
This implies that $I(\mu) > -\infty$. 
This completes the proof. 
\end{proof}
\begin{lemma} \label{L2subl-thm2-2}
Let $\{u_{n}\} \subset X$ satisfy
 $\sup_{n \in \mathbb{N}} \|u_{n}\|_{X} < \infty$ and 
$\inf_{n \in \mathbb{N}}\|u_{n}\|_{L^{p+1}} > \delta_{0}$ for some 
$\delta_{0} >0$. 
Then, there exist a sub-sequence of $\{u_{n}\}$ 
(we still denote it by the same letter), a sequence $\{(x_{n}, y_{n})\} 
\subset \mathbb{R}^{2}$ and $u_{\infty} \in X \setminus \{0\}$ such that 
$$u_{n}(\cdot + x_{n}, \cdot + y_{n}) \rightharpoonup u_{\infty} \quad {\rm in} \  X \quad {\rm as} \ n \to \infty .$$ 
\end{lemma}
\begin{proof}
By the H\"{o}lder inequality, we have 
\[
\delta_{0} < \|u_{n}\|_{L^{p+1}} \leq 
\|u_{n}\|_{L^{2}}^{\theta} \|u_{n}\|_{L^{\frac{10}{3}}}^{1-\theta} 
\lesssim \|u_{n}\|_{L^{\frac{10}{3}}}^{1-\theta},  
\]
where $\theta = \frac{7-3p}{2(p+1)}$. 
Therefore, there exists $\delta_{1} >0$, which is independent on 
$n \in \mathbb{N}$, such that 
$\|u_{n}\|_{L^{\frac{10}{3}}} \geq \delta_{1}$ for all $n \in \mathbb{N}$. 
We set 
\[
\boQ_{k, m} := (k, k +1) \times (m, m + 1)
\]
for all $(k, m) \in \mathbb{Z}^{2}$. 
Then, by the Gagliardo-Nirenberg inequality \eqref{GN-ineq}, we have 
\begin{equation*}
\begin{split}
\|u_{n}\|_{L^{\frac{10}{3}}(\boQ_{k, m})}^{\frac{10}{3}} 
& \leq C_{GN} \|u_{n}\|_{L^{2}(\boQ_{k, m})}^{\frac{4}{3}} 
\||D_{y}|^{\frac{1}{2}} u_{n}\|_{L^{2}(\boQ_{k, m})}^{\frac{4}{3}} 
\|\partial_{x} u_{n}\|_{L^{2}(\boQ_{k, m})}^{\frac{2}{3}} \\
& \leq C \|u_{n}\|_{L^{2}(\boQ_{k, m})}^{\frac{4}{3}} 
\left(\|\partial_{x} u\|_{L^{2}(\boQ_{k, m})}^{2} 
+ \||D_{y}|^{\frac{1}{2}} u\|_{L^{2}(\boQ_{k, m})}^{2} \right). 
\end{split}
\end{equation*}
Taking a sum on $(m, n) \in \mathbb{Z}^{2}$, we have 
\begin{equation*}
\begin{split}
\delta_{1} < 
\|u_{n}\|_{L^{\frac{10}{3}}(\mathbb{R}^{2})}^{\frac{10}{3}} 
& \leq C \sum_{(m, n) \in \mathbb{Z}^{2}} 
\|u_{n}\|_{L^{2}(\boQ_{k, m})}^{\frac{4}{3}} 
\left(\|\partial_{x} u\|_{L^{2}(\boQ_{k, m})}^{2} 
+ \||D_{y}|^{\frac{1}{2}} u\|_{L^{2}(\boQ_{k, m})}^{2} \right) \\
& \leq C \sup_{(k, m) \in \mathbb{Z}^{2}} 
\|u_{n}\|_{L^{2}(\boQ_{k, m})}^{\frac{4}{3}} 
\left(\|\partial_{x} u\|_{L^{2}(\mathbb{R}^{2})}^{2} 
+ \||D_{y}|^{\frac{1}{2}} u\|_{L^{2}(\mathbb{R}^{2})}^{2} \right)\\
& \leq C \sup_{(k, m) \in \mathbb{Z}^{2}} 
\|u_{n}\|_{L^{2}(\boQ_{k, m})}^{\frac{4}{3}}. 
\end{split}
\end{equation*}
Thus, there exist $(k_{n}, m_{n}) \in \mathbb{Z}^{2}$ 
and $\delta_{2} > 0$ such that 
$\delta_{2} \leq \|u_{n}\|_{L^{2}(\boQ_{k_{n}, m_{n}})}$. 
If we put $\widetilde{u}_{n}(\cdot, \cdot) = 
u_{n}(\cdot + k_{n}, \cdot + m_{n})$, we see that 
$\lim_{n \to \infty} \|\widetilde{u}_{n}\|_{X} < \infty$ 
and $\|\widetilde{u}_{n}\|_{L^{2}(\boQ_{0, 0})} \geq \delta_{2}$. 
Therefore,  
there exist 
a sub-sequence of $\{u_{n}\}$ (we still denote it by the same letter) and 
$u_{\infty} \in X$ satisfying 
$\widetilde{u}_{n} \rightharpoonup u_{\infty}$ in $X$ as $n \to \infty$. 
Moreover, by the Rellich theorem (see e.g. \cite{Park}), 
we have $\|u_{\infty}\|_{L^{2}(\boQ_{0, 0})} \geq \delta_{2}$. 
Therefore, we see that $u_{\infty} \neq 0$. 
\end{proof}
\begin{lemma} \label{L2subl-thm2-3}
Let $\theta > 1$ and $\mu > 0$. 
Then, we have $I(\theta \mu) < \theta I(\mu)$. 
\end{lemma}
\begin{proof}
We take $u \in X$ satisfying $\mathcal{H}(u) < I(\mu)/2 \; (<0)$ 
and $\mathcal{M}(u) = \mu$. 
It follows that 
\begin{equation*} 
\frac{1}{2} I(\mu) \geq \mathcal{H}(u) \geq - \frac{1}{p+1} 
\|u\|_{L^{p+1}}^{p+1}. 
\end{equation*}
This yields that 
\begin{equation} \label{eq2-2}
\|u\|_{L^{p+1}}^{p+1}  >0. 
\end{equation} 
From the definition of $I(\mu), \theta > 1$ and \eqref{eq2-2}, we have 
\begin{equation*}
\begin{split}
I(\theta \mu)
& \leq \mathcal{H}(\sqrt{\theta}u) \\
& = \frac{\theta}{2} \left(\|\partial_{x} u\|_{L^{2}}^{2} 
+ \||D_{y}|^{\frac{1}{2}} u\|_{L^{2}}^{2} \right) - \frac{\theta^{\frac{p+1}{2}}}{p+1}
\|u\|_{L^{p+1}}^{p+1} \\
& < \frac{\theta}{2} \left(\|\partial_{x} u\|_{L^{2}}^{2} 
+ \||D_{y}|^{\frac{1}{2}} u\|_{L^{2}}^{2} \right) - \frac{\theta}{p+1} \|u\|_{L^{p+1}}^{p+1} 
= \theta \mathcal{H}(u).  
\end{split}
\end{equation*}
Taking an infimum on $u$, we have 
$I(\theta \mu) < \theta I(\mu)$.  
\end{proof}

\begin{lemma} \label{L2subl-thm2-4}
Suppose that $\{u_{n}\} \subset X$ satisfies 
$\lim_{n \to \infty} \mathcal{M}(u_{n}) = \mu$ and 
$\lim_{n \to \infty} \mathcal{H}(u_{n}) = I(\mu)$. 
Then, there exists $\{(x_{n}, y_{n})\} \subset \mathbb{R}^{2}$ 
such that $\{u_{n}(\cdot + x_{n}, \cdot + y_{n})\}$ is relatively compact 
in $X$. 
In particular, $\Sigma(\mu)$ is not empty. 
\end{lemma}
\begin{proof}
It follows from \eqref{eq2-1} that 
$\{u_{n}\}$ is bounded in $X$. 
This together with \eqref{eq2-2} and Lemma \ref{L2subl-thm2-2} 
yields that there exists a sub-sequence of $\{u_{n}\}$ 
(we still denote it by the same letter), 
a sequence $\{(x_{n}, y_{n})\} \subset \mathbb{R}^{2}$ 
and $u_{\infty} \in X \setminus \{0\}$ 
such that  
$u_{n}(\cdot + x_{n}, \cdot + y_{n})\rightharpoonup  u_{\infty}$ in $X$ as $n \to \infty$. 
We shall show that $\mathcal{M}(u_{\infty})= \mu$. 
We put $\widetilde{u}_{n}(\cdot, \cdot) = u_{n}(\cdot + x_{n}, \cdot + y_{n})$. 
From the weak lower semi-continuity, we have 
\[
\mathcal{M}(u_{\infty}) \leq \liminf_{n \to \infty} 
\mathcal{M}(u_{n}) 
= \mu. 
\]
Suppose the contrary that $\mathcal{M}(u_{\infty}) < \mu$. 
From the weak convergence, we obtain 
\begin{align}
\lim_{n \to \infty} \left\{ 
\|\widetilde{u}_{n}\|_{X}^{2} - 
\|\widetilde{u}_{n} - u_{\infty}\|_{X}^{2} - 
\|u_{\infty}\|_{X}^{2}
\right\} = 0, 
\label{eq2-3}
\\
\lim_{n \to \infty} \left\{ 
\mathcal{M}(\widetilde{u}_{n}) - 
\mathcal{M}(\widetilde{u}_{n} - u_{\infty}) - 
\mathcal{M}(u_{\infty})
\right\} = 0
\label{eq2-4}
\end{align}
and by Lemma \ref{16/03/25/01:03}, we have 
\begin{equation} \label{eq2-5}
\lim_{n \to \infty} \left\{ 
\mathcal{H}(\widetilde{u}_{n}) - 
\mathcal{H}(\widetilde{u}_{n} - u_{\infty}) 
- \mathcal{H}(u_{\infty})
\right\} = 0. 
\end{equation}
We set $\nu_{n} = \mathcal{M}(\widetilde{u}_{n} - u_{\infty})$ and 
$\mu_{\infty} = \mathcal{M}(u_{\infty})$. 
Since the sequence $\{\nu_{n}\} \subset \mathbb{R}_{\geq 0}$ is bounded, 
there exist a sub-sequence of $\{\nu_{n}\}$ 
(we still denote it by the same letter) and 
$\nu_{\infty} \geq 0$ 
such that $\lim_{n \to \infty} \nu_{n} = \nu_{\infty}$. 
By \eqref{eq2-4}, we obtain 
$\mu = \nu_{\infty} + \mu_{\infty}$. 
From \eqref{eq2-5}, Lemma \ref{L2subl-thm2-3} 
with $\theta = \mu/\nu_{n}$ and $\mu/\mu_{\infty}$ and 
$\lim_{n \to \infty} \nu_{n} = \nu_{\infty}$, we have 
\begin{equation*}
\begin{split}
I(\mu) 
& \geq \lim_{n \to \infty} I(\nu_{n}) + I(\mu_{\infty}) \\
& > \lim_{n \to \infty} \frac{\nu_{n}}{\mu} I(\mu) 
+ \frac{\mu_{\infty}}{\mu} I(\mu) \\
& = \frac{\nu_{\infty}}{\mu} I(\mu) + \frac{\mu_{\infty}}{\mu} I(\mu) 
= I(\mu), 
\end{split}
\end{equation*}
which is a contradiction. 
Therefore, we obtain $\mathcal{M}(u_{\infty}) = \mu$. 

It from the definition of $I(\mu)$ that 
$I(\mu) \leq \mathcal{H}(u_{\infty})$. 
From $\mathcal{M}(u_{\infty})= \mu$ and \eqref{eq2-4}, 
we obtain $\lim_{n \to \infty} \mathcal{M}(\widetilde{u}_{n} - u_{\infty}) = 0$. 
This together with the boundedness of $\{\widetilde{u}_{n}\}$ in $X$ yields that 
$\lim_{n \to \infty} \|\widetilde{u}_{n} - u_{\infty}\|_{L^{p+1}} = 0$. 
Moreover, it follows from the weak lower semi-continuity that 
\[
\mathcal{H}(u_{\infty}) \leq \liminf_{n \to \infty}\mathcal{H}
(\widetilde{u}_{n}) = I(\mu). 
\]
Thus, we see that $\mathcal{H}(u_{\infty}) = I(\mu)$. 
Namely, $u_{\infty} \in \Sigma(\mu)$. 
Moreover, we obtain 
\begin{equation*}
\begin{split}
\lim_{n \to \infty} 
\frac{1}{2} \left(\|\partial_{x} \widetilde{u}_{n}\|_{L^{2}}^{2} 
+ \||D_{y}|^{\frac{1}{2}} \widetilde{u}_{n}\|_{L^{2}}^{2} \right)
& = \lim_{n \to \infty} \left\{\mathcal{H}(\widetilde{u}_{n}) 
+ \frac{1}{p+1} \|\widetilde{u}_{n}\|_{L^{p+1}}^{p+1} \right\} \\
& = \mathcal{H}(u_{\infty}) + \frac{1}{p+1} \|u_{\infty}\|_{L^{p+1}}^{p+1} \\
& = 
\frac{1}{2} \left(\|\partial_{x} u_{\infty}\|_{L^{2}}^{2} 
+ \||D_{y}|^{\frac{1}{2}} u_{\infty}\|_{L^{2}}^{2} \right). 
\end{split}
\end{equation*}
This implies that $\lim_{n \to \infty} \widetilde{u}_{n} = u_{\infty}$ in $X$. 
This completes the proof. 
\end{proof}
We are now in a position to prove Theorem \ref{L2subl-thm2-0}. 
\begin{proof}[Proof of Theorem \ref{L2subl-thm2-0}]
From Lemma \ref{L2subl-thm2-4}, we see that $\Sigma(\mu)$ is non-empty. 
Thus, it is enough to show the stability of $\Sigma(\mu)$. 
We shall show this by contradiction. 
Suppose the contrary that $\Sigma(\mu)$ is not stable. 
Then, there exists $\varepsilon_{0} > 0$, a sequence $\{\psi_{0, n}\} 
\subset X$ and $\{t_{n}\} \subset \mathbb{R}$ such that 
\begin{align}
& \lim_{n \to \infty} \inf_{u \in \Sigma(\mu)} 
\|\psi_{0, n} - u\|_{X} = 0, \label{eq2-6} \\
& \inf_{u \in \Sigma(\mu)} \|\psi_{n}(t_{n}) - u\|_{X} \geq \varepsilon_{0}, 
\label{eq2-7}
\end{align}
where $\psi_{n}(t)$ is a solution to \eqref{WS} with 
$\psi_{n}(0) = \psi_{0, n}$. 
From \eqref{eq2-6} and the conservation laws \eqref{conserv}, we have 
\begin{align*}
& \lim_{n \to \infty} 
\mathcal{M}(\psi_{n}(t_{n})) = 
\lim_{n \to \infty} \mathcal{M}(\psi_{0, n}) 
= \mu, \\
& \lim_{n \to \infty} \mathcal{H}(\psi_{n}(t_{n})) 
= \lim_{n \to \infty} \mathcal{H}(\psi_{0, n}) = I(\mu). 
\end{align*}
By Lemma \ref{L2subl-thm2-4}, there exist 
a sub-sequence of $\{\psi_{n}(t_{n})\}$ (we still denote it by the 
same letter), $\{(x_{n}, y_{n})\} \subset \mathbb{R}^{2}$ and 
$\psi_{\infty} \in X$ such that 
\[
\lim_{n \to \infty} \psi_{n}(t_{n}, \cdot + x_{n}, \cdot + y_{n}) 
= \psi_{\infty} \qquad 
\mbox{in $X$}.  
\]
This yields that $\psi_{\infty} \in \Sigma(\mu)$, which 
contradicts \eqref{eq2-7}. 
This completes the proof. 
\end{proof}

\section{Instability result}\label{section-instability}
This section is devoted to the proof of Theorem \ref{instability}. 
From the result of Grillakis, Shatah and Strauss
~\cite{Grillakis-Shatah-Strauss}, we can easily obtain the following: 
\begin{lemma}[Lemma 3.2 of Grillakis, Shatah 
and Strauss~\cite{Grillakis-Shatah-Strauss}]
\label{thm3-1}
There exist $\varepsilon > 0$ and a $C^{2}$ map 
$\theta: U_{\varepsilon}(Q_{\omega}) \to \mathbb{R}/2\pi \mathbb{Z}$ 
such that for all $u \in U_{\varepsilon}(Q_{\omega})$ and 
all $s \in \mathbb{R}/ 2\pi \mathbb{Z}$, 
\begin{enumerate}
\item[\rm (i)]
$\|e^{i \theta(u)}u - Q_{\omega}\|_{X} \leq 
\|e^{i s}u - Q_{\omega}\|_{X}$, 
\item[\rm (ii)]
$(e^{i \theta(u)}u, i Q_{\omega})_{X} = 0$, 
\item[\rm (iii)]
$\theta(e^{is} u) = \theta(u) - s$, 
\item[\rm (iv)]
$
\theta^{\prime}(u) = 
- \frac{i (-\partial_{xx} + |D_{y}| + 1)^{-1} e^{-i \theta(u)} Q_{\omega}}
{(Q_{\omega}, e^{i \theta(u)}u)_{X}}, 
$
\item[\rm (v)]
$i \theta^{\prime}(u)$ is a $C^{1}$ function from 
$U_{\varepsilon}(Q_{\omega})$ into $X$,  
\end{enumerate}
where 
\[
U_{\varepsilon}(Q_{\omega}) = \left\{
u \in X \mid \inf_{\theta \in \mathbb{R}} 
\|u - e^{i \theta}Q_{\omega}\|_{X} < \varepsilon
\right\}
\]
\end{lemma}
From the scale invariance of the equation \eqref{WS}, 
we can show the following:
\begin{lemma} \label{thm3-1-1}
Assume that $7/3 < p < 5$. 
Putting
\begin{equation} \label{eq3-2}
\psi_{\omega} = \frac{3}{4} Q_{\omega} 
+ \frac{1}{2}x \partial_{x} Q_{\omega} + y \partial_{y} Q_{\omega},  
\end{equation}
we have 
$\langle \mathcal{S}_{\omega}^{\prime \prime}(Q_{\omega})\psi_{\omega}, 
\psi_{\omega}\rangle < 0$ and 
$\langle Q_{\omega}, \psi_{\omega} \rangle = 0$. 
\end{lemma}
\begin{proof}
For each $\lambda>0$, we see that 
\[
\mathcal{S}_{\omega}(T_{\lambda}u) 
= \frac{\lambda}{2} \left(\|\partial_{x} u\|_{L^{2}}^{2} 
+ \||D_{y}|^{\frac{1}{2}} u\|_{L^{2}}^{2} \right) + \frac{\omega}{2} \|u\|_{L^{2}}^{2} 
- \frac{\lambda^{\frac{3}{4}(p-1)}}{p+1} \|u\|_{L^{p+1}}^{p+1},  
\]
where $T_{\lambda}$ is the scaling operator defined by 
\eqref{L2scale}. 
Then, we can compute 
\[
\frac{d^{2}}{d \lambda^{2}} \mathcal{S}_{\omega}
(T_{\lambda}u)|_{\lambda = 1} 
= - \frac{3(p-1)(3p-7)}{16(p+1)} \|u\|_{L^{p+1}}^{p+1} < 0. 
\]
We note that $\psi_{\omega} \in X$ 
(see Proposition \ref{regularity} below). 
This together with $\mathcal{S}_{\omega}^{\prime}(Q_{\omega}) = 0$, 
implies 
\[
\frac{d^{2}}{d \lambda^{2}} \mathcal{S}_{\omega}
(T_{\lambda}Q_{\omega})|_{\lambda = 1} 
= \langle \mathcal{S}_{\omega}^{\prime\prime}(Q_{\omega})
T_{\lambda}Q_{\omega}|_{\lambda =1}, T_{\lambda}Q_{\omega}|_{\lambda =1} \rangle 
= \langle \mathcal{S}_{\omega}^{\prime\prime}(Q_{\omega})
\psi_{\omega}, \psi_{\omega} \rangle.  
\]
Since $\|T_{\lambda} Q_{\omega}\|_{L^{2}}^{2} = \|Q_{\omega}\|_{L^{2}}^{2}$, 
we can easily check that $\langle Q_{\omega}, \psi_{\omega} \rangle = 0$. 
Thus, we obtain the desired result. 
\end{proof}

For $u \in U_{\varepsilon}(Q_{\omega})$, 
we define 
\begin{equation} \label{eq3-1}
A(u) = - \langle e^{i \theta(u)}u, i \psi_{\omega} \rangle.  
\end{equation}
\begin{lemma} \label{thm3-2}
Let $A(u)$ be the functional defined by \eqref{eq3-1}. 
Then, we have 
\begin{enumerate}
\item[\rm (i)]
$A(e^{is}u) = A(u)$ for all $s \in \mathbb{R}$, 
\item[\rm (ii)]
$R(i A^{\prime}(u)) \subset X$, 
\item[\rm (iii)]
$i A^{\prime}(Q_{\omega}) = - \psi_{\omega}$, 
\item[\rm (iv)]
$\langle u, i A^{\prime}(u) \rangle = 0$. 
\end{enumerate}
\end{lemma}
\begin{proof}
{\rm (i)}.
It follows from Lemma \ref{thm3-1} {\rm (iii)} that 
\[
A(e^{is}u) = 
- \langle e^{i \theta(e^{is}u)} e^{is}u, i \psi_{\omega} \rangle 
= - \langle e^{i \theta(u) -is} e^{is}u, i \psi_{\omega} \rangle 
= - \langle e^{i \theta(u)} u, i \psi_{\omega} \rangle 
= A(u). 
\]

{\rm (ii)}.
From \eqref{eq3-1}, we see that 
\[
\langle A^{\prime}(u), v \rangle 
= - \langle e^{i \theta(u)}v, i \psi_{\omega} \rangle 
- \langle e^{i \theta(u)}u, i \psi_{\omega} \rangle
\langle i \theta^{\prime}(u), v \rangle 
\]
for $v \in X$. 
This yields that 
\begin{equation} \label{eq3-3}
A^{\prime}(u) 
= - i e^{-i \theta(u)} \psi_{\omega} - 
\langle e^{i \theta(u)}u, \psi_{\omega} \rangle \theta^{\prime}(u).
\end{equation}
This together with Lemma \ref{thm3-1} (iv) 
yields that $A^{\prime}(u) \in X$. 

{\rm (iii)}. 
Substituting $Q_{\omega}$ for $u$ in \eqref{eq3-3}, we have 
\[
A^{\prime}(Q_{\omega}) = 
-i e^{i \theta(Q_{\omega})} \psi_{\omega} 
- \langle e^{i \theta(Q_{\omega})}Q_{\omega}, \psi_{\omega} \rangle 
\theta^{\prime}(Q_{\omega}) 
= -i \psi_{\omega} - \langle Q_{\omega}, \psi_{\omega} \rangle 
\theta^{\prime}(Q_{\omega})
= - i \psi_{\omega}, 
\]
where we have used the fact that $\theta(Q_{\omega}) = 0$ and 
$\langle Q_{\omega}, \psi_{\omega} \rangle = 0$. 
It follows that $i A^{\prime}(Q_{\omega}) = \psi_{\omega}$. 

{\rm (iv)}.
From {\rm (i)}, we obtain 
\[
0 = \frac{d}{d s} A(e^{is} u)\biggl|_{s=0} = 
\langle A^{\prime}(u), i u \rangle.  
\]
This yields that $\langle i A^{\prime}(u), u \rangle = 0$. 
\end{proof}
Next, we consider the following differential equation: 
\begin{equation} \label{eq3-4}
\frac{d R}{d \lambda} = -i A^{\prime}(R), \qquad 
R(0) = v \in U_{\varepsilon}(Q_{\omega}). 
\end{equation}
We note that thanks to Lemma \ref{thm3-2}, 
there exist $\lambda_{0} > 0$ and 
a unique solution $R(\lambda, v)$ 
to the equation \eqref{eq3-4} for $|\lambda| < \lambda_{0}$. 
We see that 
\begin{align}
& e^{is}R(\lambda, v) = R(\lambda, e^{is} v), 
\label{eq3-5} \\
& \frac{d}{d \lambda} \mathcal{M}(R(\lambda, v)) 
= \langle R(\lambda, v), - i A^{\prime}(R(\lambda, v)) \rangle 
= 0, 
\label{eq3-6} \\
& \frac{d}{d \lambda} R(\lambda, Q_{\omega})|_{\lambda =0} 
= - i A^{\prime}(Q_{\omega}) = \psi_{\omega}.  
\label{eq3-7}
\end{align}
Indeed, we find that 
\[
e^{i s}R(0, v) = e^{is} v = R(0, e^{is} v)
\]
and it follows from \eqref{eq3-3} and $\theta^{\prime}(e^{is}u) 
= e^{i s} \theta^{\prime}(u)$ that 
\[
\frac{d}{d \lambda} (e^{i s}R(\lambda, v)) 
= - i e^{is} A^{\prime}(R(\lambda, v)) = - iA^{\prime} 
(e^{is} R(\lambda, v)). 
\]
From the uniqueness of the solution, we have \eqref{eq3-5}. 
\eqref{eq3-6} immediately follows from Lemma \ref{thm3-2} (iv). 
By Lemma \ref{thm3-2} (iii), we see that \eqref{eq3-7} holds.

We put 
\[
\mathcal{P}(v) = 
\langle \mathcal{S}_{\omega}^{\prime}(v), - i A^{\prime}(v) \rangle. 
\]
Then, we can find that 
$\mathcal{P}(e^{i \theta}v) = \mathcal{P}(v)$ for any $\theta \in \mathbb{R}$. 
We shall show the following: 
\begin{lemma} \label{thm3-3}
There exists a $C^{1}$ functional 
\[
\Lambda 
: \left\{v \in U_{\varepsilon}(Q_{\omega}) \mid 
\mathcal{M}(v) = \mathcal{M}(Q_{\omega}) \right\} \to \mathbb{R}
\]
such that 
\begin{equation} \label{eq3-8}
\mathcal{S}_{\omega}(Q_{\omega}) < 
\mathcal{S}_{\omega}(v) + \mathcal{P}(v) \Lambda(v)
\end{equation}
for all $v \in U_{\varepsilon}(Q_{\omega})$ with 
$\mathcal{M}(v) = \mathcal{M}(Q_{\omega})$ 
and $v \not\in \left\{e^{is}Q_{\omega} \mid s \in \mathbb{R} 
\right\}$. 
\end{lemma}
\begin{proof}
By the Taylor expansion, we have 
\[
\mathcal{S}_{\omega}(R(\lambda, v)) 
= \mathcal{S}_{\omega}(v) + \mathcal{P}(v) \lambda 
+ \frac{1}{2} \mathcal{R}(R(\tau \lambda, v)) \lambda^{2}, 
\]
where $\tau \in (0, 1)$ and 
\[
\mathcal{R}(v) = \langle \mathcal{S}_{\omega}^{\prime\prime}(v)i A^{\prime}(v), 
i A^{\prime}(v) \rangle 
- \langle \mathcal{S}_{\omega}^{\prime}(v), 
i A^{\prime\prime}(i A^{\prime}(v)) \rangle.  
\]
By Lemma \ref{thm3-1-1} and Lemma \ref{thm3-2} {\rm (iii)}, we have 
\[
\mathcal{R}(Q_{\omega}) = 
\langle \mathcal{S}_{\omega}^{\prime\prime}(Q_{\omega}) \psi_{\omega}, 
\psi_{\omega} \rangle < 0. 
\]
Thus, for sufficiently small $\varepsilon_{0}>0$ and $\lambda_{0} > 0$,  
we obtain 
\begin{equation} \label{eq3-9}
\mathcal{S}_{\omega}(R(\lambda, v)) < \mathcal{S}_{\omega}(v) 
+ P(v) \lambda
\end{equation}
for all $|\lambda| < \lambda_{0}$ and $v \in B(Q_{\omega}, \varepsilon_{0})$. 
On the other hand, note that 
$\mathcal{N}_{\omega}(R(\lambda, v))|_{(\lambda, v)= (0, Q_{\omega})} = 0$ 
and since $i A^{\prime}(Q_{\omega}) = -\psi_{\omega}$, we have 
\begin{equation*}
\begin{split}
\frac{\partial}{\partial \lambda} \mathcal{N}_{\omega}(R(\lambda, v))
|_{(\lambda, v) = (0, Q_{\omega})} 
= \langle \mathcal{N}_{\omega}^{\prime}(Q_{\omega}), 
- i A^{\prime}(Q_{\omega}) \rangle 
& = \langle \mathcal{N}_{\omega}^{\prime}(Q_{\omega}), 
\psi_{\omega} \rangle \\
& = - \frac{3(p-1)}{4(p+1)} \|Q_{\omega}\|_{L^{p+1}}^{p+1} \neq 0.  
\end{split}
\end{equation*}
Thus, by the implicit function theorem, there exist 
sufficiently small $\varepsilon >0$ 
and a map $\Lambda(v): B(Q_{\omega}, \varepsilon) \to \mathbb{R}$ 
such that $\mathcal{N}_{\omega}(R(\Lambda(v), v)) = 0$. 
From the variational characterization, we have 
\[
\mathcal{S}_{\omega}(Q_{\omega}) \leq 
\mathcal{S}_{\omega}(R(\Lambda(v), v)). 
\]
This together with \eqref{eq3-8} yields that 
\begin{equation} \label{eq3-10}
\mathcal{S}_{\omega}(Q_{\omega}) <  
\mathcal{S}_{\omega}(v) + \mathcal{P}(v) \Lambda(v)
\end{equation}
for all $v \in B(Q_{\omega}, \varepsilon)$. 
By the gauge invariance of the functional 
$\mathcal{S}_{\omega}$ and $\mathcal{P}$, we see that 
\eqref{eq3-10} holds for all 
$v \in U_{\varepsilon}(Q_{\omega})$. 
\end{proof}
We put 
\begin{equation} \label{exit-time}
T_{\varepsilon}(\psi_{0}) 
= \sup\left\{T>0 \mid \psi(t) \in U_{\varepsilon}(Q_{\omega})\;  
\mbox{for all $t \in [0, T)$}\right\}, 
\end{equation}
where $\psi$ is a solution to \eqref{WS} with 
$\psi|_{t=0} = \psi_{0} \in U_{\varepsilon}(Q_{\omega})$. 
\begin{lemma} \label{thm3-4}
We set 
\begin{align*}
& S_{+} := 
\left\{ 
u \in U_{\varepsilon}(Q_{\omega}) \mid 
\mathcal{S}_{\omega}(u) < \mathcal{S}_{\omega}(Q_{\omega}), \;  
\mathcal{P}(u) > 0
\right\}, \\
& S_{-} := 
\left\{ 
u \in U_{\varepsilon}(Q_{\omega}) \mid 
\mathcal{S}_{\omega}(u) < \mathcal{S}_{\omega}(Q_{\omega}), \; 
\mathcal{P}(u) < 0
\right\}. 
\end{align*}
Then, the sets $S_{\pm}$ are non-empty invariant under the flow 
of \eqref{WS}. 
Moreover, for any $\psi_{0} \in S_{+} \cup S_{-}$, 
there exists $\delta_{0} = \delta_{0}(\psi_{0}) >0$ such that 
$|\mathcal{P}(\psi(t))| \geq \delta_{0}$ 
for all $t \in [0, T_{\varepsilon}(\psi_{0}))$. 
\end{lemma}
\begin{proof}
First, we shall show that $S_{\pm}$ is non-empty. 
Since $\mathcal{P}(Q_{\omega}) = 0$, 
we have, by \eqref{eq3-9}, that 
\begin{equation} \label{eq3-12}
\mathcal{S}_{\omega}(R(\lambda, Q_{\omega})) 
< \mathcal{S}_{\omega}(Q_{\omega}). 
\end{equation}
Moreover, applying \eqref{eq3-9} for $v = R(\nu, Q_{\omega})$, we obtain
\[
\mathcal{S}_{\omega}(R(\lambda + \nu, Q_{\omega})) 
= \mathcal{S}_{\omega}(R(\lambda, R(\nu, Q_{\omega}))) 
\leq \mathcal{S}_{\omega}(R(\lambda, Q_{\omega})) 
+ \mathcal{P}(R(\lambda, Q_{\omega})) \nu. 
\]
Taking $\nu = -\lambda$, we have 
\[
\mathcal{S}_{\omega}(Q_{\omega}) \leq 
\mathcal{S}_{\omega}(R(\lambda, Q_{\omega})) 
- P(R(\lambda, Q_{\omega})) \lambda. 
\]
It follows from \eqref{eq3-12} that 
\[
0 < \mathcal{S}_{\omega}(Q_{\omega}) - 
\mathcal{S}_{\omega}(R(\lambda, Q_{\omega})) 
\leq - \mathcal{P}(R(\lambda, Q_{\omega})) \lambda. 
\]
From this, we see that 
$\mathcal{P}(R(\lambda, Q_{\omega})) >0$ for $\lambda < 0$, 
which implies that $R(\lambda, Q_{\omega}) \in S_{+}$. 
By a similar argument, we find that $R(\lambda, Q_{\omega}) 
\in S_{-}$ for $\lambda >0$.  

Next, we shall show that the sets $S_{\pm}$ are invariant. 
Let $\psi_{0} \in S_{+}$. 
Then, from the conservation laws, we have 
\[
\mathcal{S}_{\omega}(\psi(t)) = \mathcal{S}_{\omega}(\psi_{0})
 < \mathcal{S}_{\omega}(Q_{\omega}).
\] 
Moreover, it follows from \eqref{eq3-8} that 
\begin{equation} \label{eq3-11}
0 < \mathcal{S}_{\omega}(Q_{\omega}) - \mathcal{S}_{\omega}(\psi(t)) 
< \mathcal{P}(\psi(t)) \Lambda (\psi(t)) \leq 
C |\mathcal{P}(\psi(t))|. 
\end{equation}
Then, by the continuity dependence of $\mathcal{P}(\psi(t))$ 
on $t$, we see that 
$P(\psi(t)) > 0$. 
Consequently, we see that the set $S_{+}$ is invariant. 
By a similar argument, we can show that the set $S_{-}$ 
is invariant. 
The last claim follows from \eqref{eq3-11}. 
\end{proof}
We are now in position to prove Theorem \ref{instability}.
\begin{proof}[Proof of Theorem \ref{instability}]
Let $\psi_{0} \in S_{+} \cup S_{-}$. 
Then, we shall show that $T_{\varepsilon}(\psi_{0}) < \infty$, 
where $T_{\varepsilon}(\psi_{0})$ is defined by \eqref{exit-time}.
Suppose the contrary that $T_{\varepsilon}(\psi_{0}) =\infty$. 
Then, from Lemma \ref{thm3-4}, there exists $\delta_{0} > 0$ 
such that 
\begin{equation}\label{eq3-13}
\mbox
{
$\mathcal{P}(\psi(t)) \geq \delta_{0}$ \quad or \quad  
$\mathcal{P}(\psi(t)) \leq - \delta_{0}$ \qquad for all $t \in [0, \infty)$. 
}
\end{equation}
It follows from Lemma \ref{thm3-2} {\rm (iv)} that 
\[
\frac{d}{d t}A(\psi(t)) 
= \langle A^{\prime}(\psi(t)), -i \mathcal{H}^{\prime}(\psi(t)) \rangle 
= \langle A^{\prime}(\psi(t)), -i \mathcal{S}_{\omega}^{\prime}(\psi(t)) \rangle
= \mathcal{P}(\psi(t)). 
\]
This together with \eqref{eq3-13} 
yields that $\lim_{t \to \infty}|A(\psi(t))| = \infty$. 
However, from the definition \eqref{eq3-1}, we see that 
there exists $C_{1} >0$ such that 
$|A(u)| \leq C_{1}$ for all $u \in U_{\varepsilon}(Q_{\omega})$, 
which is a contradiction. 
\end{proof}

\section{Existence of traveling wave solutions}
\label{section-travelingwave}
This section is devoted to the traveling wave solutions. 
First, we shall show the existence of the traveling wave solutions. 
Note that if $v \in \mathbb{R}$ with $|v| < 1$, 
we have 
\begin{equation} \label{travel-eq-1}
\begin{split}
\langle |D_{y}|u, u \rangle - \langle i v \partial_{y} u, u \rangle 
& \geq \int_{\mathbb{R}_{\eta}} 
\left\{\int_{\mathbb{R}_{\xi}} |\xi| |\widehat{u}(\eta, \xi)|^{2} 
- |v| |\xi|  |\widehat{u}(\eta, \xi)|^{2} d\xi\right\} d\eta \\
& = (1 - |v|)\langle |D_{y}| u, u \rangle \geq 0
\end{split}
\end{equation}
for all $u \in X$. 
Thus, if we put 
\begin{equation*}
\begin{split}
\mathcal{I}_{\omega, v}(u) 
& = \mathcal{S}_{\omega, v}(u) 
- \frac{1}{p+1} \mathcal{N}_{\omega, v}(u) \\
& = \left(\frac{1}{2} - \frac{1}{p+1}\right)
\int_{\mathbb{R}^{2}} \biggl\{ 
|\partial_{x} u(x, y)|^{2} + |D_{y}| u(x, y) \overline{u(x, y)}  
- i v \partial_{y} u(x, y) \overline{u(x, y)} \\
& \hspace{4cm} + \omega |u(x, y)|^{2} \biggl\} dxdy,
\end{split}
\end{equation*}
where $\mathcal{N}_{\omega, v}$ defined in \eqref{def:N-omega-v}. Then, we obtain $\mathcal{I}_{\omega, v}(u) \geq 0$ 
for all $u \in X$. 
First, we prepare the following: 
\begin{lemma} 
\label{L2subl-thm2-1-1}
Let $m_{\omega, v}$ be the minimization problem defined in \eqref{def:m-omega-v}. For any $\omega>0$ and $v \in \mathbb{R}$ with $|v|<1$, 
we have 
\begin{equation} \label{eq-0}
m_{\omega, v} = \inf \left\{
\mathcal{I}_{\omega, v}(u) \mid \mathcal{N}_{\omega, v}(u) \leq 0 \right\}.  
\end{equation}
\end{lemma}
\begin{proof}
First, denote that we obviously have
$$ m_{\omega, v} \geq \inf \left\{
\mathcal{I}_{\omega, v}(u) \mid \mathcal{N}_{\omega, v}(u) \leq 0 \right\}.  $$
It is sufficient then to show the other side inequality. Let $u \in X$ satisfy $\mathcal{N}_{\omega, v}(u) = 0$.
Then, we obtain 
\[
\mathcal{S}_{\omega, v}(u) = \mathcal{S}_{\omega, v}(u) - \frac{1}{p+1} 
\mathcal{N}_{\omega, v}(u)
= \mathcal{I}_{\omega, v}(u). 
\]
Therefore, we have 
\begin{equation} \label{eq-1}
m_{\omega, v} = \inf\left\{\mathcal{I}_{\omega, v}(u) 
\mid \mathcal{N}_{\omega, v}(u) 
= 0 \right\}. 
\end{equation}
Suppose that $\mathcal{N}_{\omega, v}(u) < 0$. 
Then, there exists $t_{0} \in (0, 1)$ such that 
$\mathcal{N}_{\omega, v}(t_{0} u) = 0$. 
This and \eqref{eq-1} lead to 
\[
m_{\omega, v} \leq \mathcal{I}_{\omega, v}(t_{0}u) < 
\mathcal{I}_{\omega, v}(u). 
\]
Taking an infimum on $u \in X$, we obtain \eqref{eq-0}. 
This completes the proof
\end{proof}
We are now in position to prove Theorem \ref{travel-thm1}. 
\begin{proof}[Proof of Theorem \ref{travel-thm1}]
Let $\{u_{n}\} \subset X$ be a minimizing sequence 
for $m_{\omega, v}$. 
Then, for sufficiently large $n \in \mathbb{N}$, we have 
\[
1 + m_{\omega, v} \geq \mathcal{S}_{\omega, v}(u_{n}) 
- \frac{1}{p+1} \mathcal{N}_{\omega, v}(u_{n}) = \mathcal{I}_{\omega, v}(u_{n}). 
\]
Thus, we see that $\{u_{n}\}$ is bounded in $X$. 
Moreover, it follows from $\mathcal{N}_{\omega, v}(u_{n}) = 0$ and 
the Sobolev inequality that there exists a constant $C_{1} > 0$ such that 
\begin{equation*}
\begin{split}
C_{1} \|u_{n}\|_{L^{p+1}}^{2}  \leq 
\int_{\mathbb{R}^{2}} \biggl\{ & |\partial_{x} u_{n}(x, y)|^{2} + |D_{y}| u_{n}(x, y) \overline{u(x, y)}  
- i v \partial_{y} u_{n}(x, y) \overline{u(x, y)} \\
& + \omega |u_{n}(x, y)|^{2} \biggl\} dxdy  = \|u_{n}\|_{L^{p+1}}^{p+1}. 
\end{split}
\end{equation*}
Thus, we have $\|u_{n}\|_{L^{p+1}} \geq C_{1}^{1/(p-1)}$. 
It follows from Lemma \ref{L2subl-thm2-2} that 
there exist $\{(x_{n}, y_{n})\} \subset \mathbb{R}^{2}$, a sub-sequence 
of $\{u_{n}\}$ (we shall denote it by the same letter) and 
$Q_{\omega, v} \in X \setminus\{0\}$ such that 
$u_{n}(\cdot + x_{n}, \cdot + y_{n})$ converges to $Q_{\omega, v}$ 
weakly in $X$ as $n$ tends to infinity. 
We put $v_{n}(\cdot, \cdot) = u_{n}(\cdot + x_{n}, \cdot + y_{n})$. 
Then, by Lemma \ref{16/03/25/01:03}, we have 
\begin{align}
& \mathcal{I}_{\omega, v}(v_{n}) - \mathcal{I}_{\omega, v}(v_{n} - Q_{\omega, v}) 
- \mathcal{I}_{\omega, v}(Q_{\omega, v}) \to 0 \qquad 
\mbox{as $n \to \infty$}, \label{eq4-1} \\
& \mathcal{N}_{\omega, v}(v_{n}) - \mathcal{N}_{\omega, v}(v_{n} - Q_{\omega, v}) 
- \mathcal{N}_{\omega, v}(Q_{\omega, v}) \to 0 \qquad 
\mbox{as $n \to \infty$}. \label{eq4-2}  
\end{align}  
Suppose that $\mathcal{N}_{\omega, v}(Q_{\omega, v}) >0$. 
Then, from \eqref{eq4-2} and $\mathcal{N}_{\omega, v}(v_{n}) = 
\mathcal{N}_{\omega, v}(u_{n}) = 0$, we obtain 
$\mathcal{N}_{\omega, v}(v_{n} - Q_{\omega, v}) < 0$ for sufficiently large 
$n \in \mathbb{N}$. 
For such $n \in \mathbb{N}$, there exists $t_{n} \in  (0, 1)$ such that 
$\mathcal{N}_{\omega, v}(t_{n}(v_{n} - Q_{\omega, v})) = 0$. 
From the definition of $m_{\omega, v}$, we see that 
\[
m_{\omega, v} \leq \mathcal{I}_{\omega, v}(t_{n}(v_{n} - Q_{\omega, v})) 
< \mathcal{I}_{\omega, v} (v_{n} - Q_{\omega, v}). 
\]
This yields that $m_{\omega, v} \leq \liminf_{n \to \infty}
\mathcal{I}_{\omega, v}(v_{n} - Q_{\omega, v})$. 
This together with $Q_{\omega, v} \neq 0$, \eqref{eq4-1} and 
$\lim_{n \to \infty}\mathcal{I}_{\omega, v}(v_{n}) = m_{\omega, v}$ 
yields that 
\[
0 < \mathcal{I}_{\omega, v}(Q_{\omega, v}) \leq \limsup_{n \to \infty}
\left\{- I_{\omega, v}(v_{n}) + \mathcal{I}_{\omega, v}(v_{n} - Q_{\omega, v}) 
+ \mathcal{I}_{\omega, v}(Q_{\omega, v}) \right\} = 0, 
\]
which is a contradiction. 
Therefore, we see that $\mathcal{N}_{\omega, v}(Q_{\omega, v}) \leq 0$. 
By Lemma \ref{L2subl-thm2-1}, we have $m_{\omega, v} 
\leq \mathcal{I}_{\omega, v}(Q_{\omega, v})$. 
It follows from the weak lower semi-continuity that 
\[
m_{\omega, v} \leq \mathcal{I}_{\omega, v}(Q_{\omega, v}) \leq 
\liminf_{n \to \infty}\mathcal{I}_{\omega, v}(v_{n}) = m_{\omega, v}. 
\]
Therefore, we have $\mathcal{I}_{\omega, v}(Q_{\omega, v}) = m_{\omega, v}$. 
Then, from \eqref{eq4-1} and 
$\lim_{n \to \infty}\mathcal{I}_{\omega, v}(v_{n}) = m_{\omega, v}$, 
we have $\lim_{n \to \infty}\mathcal{I}_{\omega, v}(v_{n} - Q_{\omega, v}) = 0$, 
which implies that $v_{n}$ converges to $Q_{\omega, v}$ 
strongly in $X$ as $n$ tends to infinity. 
Therefore, we have 
\[
\mathcal{N}_{\omega, v}(Q_{\omega, v}) = 
\lim_{n \to \infty} \mathcal{N}_{\omega, v}(v_{n}) 
= \lim_{n \to \infty} \mathcal{N}_{\omega, v}(u_{n}) = 0,  
\]
which yields that $Q_{\omega, v}$ is a minimizer for $m_{\omega, v}$. 
Thus, there exists a Lagrange multiplier $\lambda \in \mathbb{R}$ 
such that 
$\mathcal{S}_{\omega, v}^{\prime}(Q_{\omega, v}) 
= \lambda \mathcal{N}_{\omega, v}^{\prime}(Q_{\omega, v})$. 
This yields that 
\[
0 = \mathcal{N}_{\omega, v}(Q_{\omega, v}) 
= \langle \mathcal{S}_{\omega, v}^{\prime}(Q_{\omega, v}), 
Q_{\omega, v} \rangle
= \lambda \langle \mathcal{N}_{\omega, v}^{\prime}
(Q_{\omega, v}), Q_{\omega, v} \rangle
= - \lambda (p-1)\|Q_{\omega, v}\|_{L^{p+1}}^{p+1}. 
\]
Since $Q_{\omega, v} \neq 0$, we have $\lambda = 0$. 
This yields that $S_{\omega, v}^{\prime}(Q_{\omega, v}) = 0$, 
that is, $Q_{\omega, v}$ is a solution to \eqref{spv}. 
Moreover, we see that for any non-trivial solution $u \in X$ to \eqref{spv}, 
$\mathcal{N}_{\omega, v}(u) = 0$. 
It follows from the definition of $m_{\omega, v}$ 
that $\mathcal{S}_{\omega, v}(Q_{\omega, v}) 
\leq \mathcal{S}_{\omega, v}(u)$. 
Thus, we find that $Q_{\omega, v}$ is a ground state to \eqref{spv}. 
\end{proof}
Next, we shall give a proof of Theorem \ref{travel-thm12}. 
It suffices to show the case where $v$ goes to $1$ because 
the other case can be proved similarly. 
\begin{proof}[Proof of Theorem \ref{travel-thm12}]
Let $\varphi \in X \setminus \{0\}$ satisfy 
$\mathop{\mathrm{supp}} \widehat{\varphi} \subset [0, \infty)$. 
This yields that 
\[
(|D_{y}| + i v \partial_{y}) \varphi = (1 - v)|D_{y}|\varphi
\] 
for $v > 0$. 
For $\lambda >0$, we set $\varphi_{\lambda}(x, y) = 
\lambda \varphi(x, \lambda^{\alpha}y)\; (\alpha > 2)$. 
Then, we have 
\begin{equation*}
\begin{split}
\mathcal{N}_{\omega, v}(\varphi_{\lambda})
& = \lambda^{2 - \alpha} \|\partial_{x} \varphi\|_{L^{2}}^{2} 
+ \lambda^{2} (1 - v)\langle |D_{y}| \varphi, \varphi \rangle 
+ \lambda^{2 - \alpha}\|\varphi\|_{L^{2}}^{2} - 
\lambda^{p+1 - \alpha} \|\varphi\|_{L^{p+1}}^{p+1} \\
& = \lambda^{2 - \alpha} 
\left( \|\partial_{x} \varphi\|_{L^{2}}^{2} + \lambda^{\alpha}(1- v)
\langle |D_{y}| \varphi, \varphi \rangle + 
\omega \|\varphi\|_{L^{2}}^{2} - \lambda^{p-1}\|\varphi\|_{L^{p+1}}^{p+1}
\right). 
\end{split}
\end{equation*}
We take $\lambda >0$ sufficiently large so that 
\[
\| \partial_{x} \varphi\|_{L^{2}}^{2} + 
\omega \|\varphi\|_{L^{2}}^{2} - \lambda^{p-1}\|\varphi\|_{L^{p+1}}^{p+1}
< - \|\varphi\|_{L^{p+1}}^{p+1}. 
\]
For $\lambda>0$, we set $v = 1- \lambda^{-\alpha}$ so that 
\[
\lambda^{2} (1 - v)\langle |D_{y}| \varphi, \varphi \rangle 
= \lambda^{2-\alpha}\langle |D_{y}| \varphi, \varphi \rangle 
< \frac{\|\varphi\|_{L^{p+1}}^{p+1}}{2} 
\]
for sufficiently large $\lambda>0$
This yields that 
\[
\mathcal{N}_{\omega, v}(\varphi_{\lambda}) < - \frac{\lambda^{2-\alpha}}{2}
\|\varphi\|_{L^{p+1}}^{p+1} < 0.  
\]
By Lemma \ref{L2subl-thm2-1}, we have 
\[
m_{\omega, v} \leq \mathcal{I}_{\omega, v}(\varphi_{\lambda}) 
= \left(\frac{1}{2} - \frac{1}{p+1}\right)
\left\{ 
\lambda^{2-\alpha}\|\partial_{x} \varphi\|_{L^{2}}^{2} 
+ \lambda^{2} (1 - v)\langle |D_{y}| \varphi, \varphi \rangle 
+ \lambda^{2 - \alpha} \|\varphi\|_{L^{2}}^{2}
\right\}. 
\]
Since $\alpha>2$, we see that 
\[
\lambda^{2-\alpha}\|\partial_{x} \varphi\|_{L^{2}}^{2}, \quad 
\lambda^{2 - \alpha} \|\varphi\|_{L^{2}}^{2} \to 0 \qquad 
\mbox{as $\lambda \to \infty$}. 
\]
and 
\[
\lambda^{2} (1 - v)\langle |D_{y}| \varphi, \varphi \rangle
= \lambda^{2-\alpha} \langle |D_{y}| \varphi, \varphi \rangle
\to 0 \qquad \mbox{as $\lambda \to \infty$}. 
\]
Therefore, we have $\lim_{v \to 1} m_{\omega, v} = 0$. 
Since 
\[
m_{\omega, v} = \mathcal{I}_{\omega, v}(Q_{\omega, v})
\geq \left(\frac{1}{2} - \frac{1}{p+1}\right)
\left\{\|\partial_{x} Q_{\omega, v}\|_{L^{2}}^{2} + 
\omega \|Q_{\omega, v}\|_{L^{2}}^{2} \right\}, 
\]
we deduce that 
$\|\partial_{x}Q_{\omega, v}\|_{L^{2}}, \|Q_{\omega, v}\|_{L^{2}} \to 0$ 
as $v \to 1$. 
This completes the proof. 
\end{proof}

\section{Cauchy problem}
\label{section-Cauchy-Problem}
In this section, we prove Theorem \ref{Cauchy}. To do so, we need several preparations. From the result of Grafakos and Oh~\cite[Theorem 1]{Grafakos-Oh}, we see that the following inequality holds: 
\begin{lemma}\label{product-chain0}
For every $0 < s < 1$, we have
\begin{equation}\label{eq-product}
\||D_{y}|^{s} (uv)\|_{L^{2}} \lesssim 
\||D_{y}|^{s}u\|_{L^{2}} \|v\|_{L^{\infty}} + 
\|u\|_{L^{\infty}} \||D_{y}|^{s} v\|_{L^{2}}
\end{equation}
for every $u, v \in \dot{H}^{s}\cap L^{\infty}(\mathbb{R})$. 
\end{lemma}
Secondly, following Tzvetkov and Visciglia~\cite[Lemma 4.1]{Tzvetkov-Visciglia}, 
we shall prove the following lemma: 
\begin{lemma} \label{product-chain}
For every $0 < s < 1$ and $1 < p< 5$, 
we obtain
\begin{equation}\label{eq-chain}
\||D_{y}|^{s} (|u|^{p-1}u)\|_{L^{2}} \lesssim 
\||D_{y}|^{s}u\|_{L^{2}} \|u\|_{L^{\infty}}^{p-1}  
\end{equation}
for every $u \in \dot{H}^{s}\cap L^{\infty}(\mathbb{R})$. 
\end{lemma}
\begin{proof}
First, we claim that the following identity holds:
\begin{equation} \label{frac-identity}
\iint_{\mathbb{R}^{2}} \dfrac{|u(y+h) - u(y)|^{2}}{|h|^{1+2s}} dydh 
= C_{*} \||D_{y}|^{s} u\|_{L^{2}}^{2},  
\end{equation}
where $C_{*} > 0$ is a constant defined by \eqref{const-fra} below.

It follows from the Plancharel theorem that 
\[
\int_{\mathbb{R}} |u(y+h) - u(y)|^{2} dy = 
\int_{\mathbb{R}} |e^{-ihy}\mathcal{F}_{y}[u](\eta) 
- \mathcal{F}_{y}[u](\eta)|^{2} d\eta,  
\]
where $\mathcal{F}_{y}[u]$ is the Fourier transformation of $u$ 
with respect to the variable $y$. 
This yields that 
\begin{equation} \label{frac-identity-1}
\begin{split}
& \quad \iint_{\mathbb{R}^{2}} \dfrac{|u(y+h) - u(y)|^{2}}{|h|^{1+2s}} dydh \\
& = \int^{\infty}_{0} |\mathcal{F}[u](\eta)|^{2} \int_{\mathbb{R}} 
\frac{|e^{i \eta h} - 1|^{2}}{|h|^{1+2s}} dh d\eta
+ \int^{0}_{-\infty} |\mathcal{F}[u](\eta)|^{2} \int_{\mathbb{R}} 
\frac{|e^{i \eta h} - 1|^{2}}{|h|^{1+2s}} dh d\eta
=: I + II. 
\end{split}
\end{equation}
We first consider $I$. 
Changing the variable $r = \eta h$, 
we have 
\[
 \int_{\mathbb{R}} 
\frac{|e^{i \eta h} - 1|^{2}}{|h|^{1+2s}} dh 
= \int_{\mathbb{R}} |e^{i r} -1|^{2} \frac{|\eta|^{2s}}{|r|^{1+2s}} dr 
= C_{*} |\eta|^{2s}, 
\]
where 
\begin{equation} \label{const-fra}
C_{*} = \int_{\mathbb{R}} \frac{|e^{ir} -1|^{2}}{|r|^{1 + 2s}} dr. 
\end{equation}
Therefore, we have 
\begin{equation} \label{frac-identity2}
I = C_{*} \int^{\infty}_{0} |\eta|^{2s}|\mathcal{F}[u](\eta)|^{2} d\eta. 
\end{equation}
Next, we consider the case $\eta < 0$. 
Changing the variable  $r = - \eta h$, 
we obtain 
\[
 \int_{\mathbb{R}} 
\frac{|e^{i \eta h} - 1|^{2}}{|h|^{1+2s}} dh 
= \int_{\mathbb{R}} |e^{- i r} -1|^{2} \frac{|\eta|^{2s}}{|r|^{1+2s}} dr 
= |\eta^{2s}| \int_{\mathbb{R}} |e^{i r} -1|^{2} \frac{1}{|r|^{1+2s}} dr 
= C_{*} |\eta|^{2s}.  
\]
This implies that 
\begin{equation} \label{frac-identity3}
II = C_{*} \int^{0}_{-\infty} |\eta|^{2s}|\mathcal{F}[u](\eta)|^{2} d\eta. 
\end{equation}
By \eqref{frac-identity-1}, \eqref{frac-identity2} and 
\eqref{frac-identity3}, 
we see that \eqref{frac-identity} holds.

It follows from \eqref{frac-identity} that 
\begin{equation*}
\begin{split}
C_{*} \||D_{y}|^{s} (|u|^{p-1}u)\|_{L^{2}}^{2} 
& = \iint_{\mathbb{R}^{2}} 
\dfrac{\bigl||u|^{p-1}u(y+h) - |u|^{p-1}u(y)\bigl|^{2}}{|h|^{1+2s}} dydh \\
& \lesssim 
\|u\|_{L^{\infty}}^{2(p-1)} 
\int_{\mathbb{R}^{2}} \dfrac{|u(y+ h) - u(y)|^{2}}{|h|^{1+2s}} dydh \\
& = C_{*} \|u\|_{L^{\infty}}^{2(p-1)}\||D_{y}|^{s} u\|_{L^{2}}^{2}. 
\end{split}
\end{equation*}
This completes the proof.

\end{proof}

Next, we recall the following Strichartz estimate stated by Keel and Tao in \cite{Keel-Tao}.
\begin{lemma}[\cite{Keel-Tao}]
\label{Lem:Keel-Tao}
Suppose that for each time $t \in \mathbb{R}$, 
we have an operator $U(t) : L^2(\mathbb{R}^2) \to L^2(\mathbb{R}^2)$ which obeys the following estimates :
\begin{itemize}
\item For all $t$ and all $f \in L^2(\mathbb{R}^2)$ we have 
\begin{equation}
\label{cond1U:L2}
\|U(t)f\|_{L^2_{x,y}} \lesssim \|f\|_{L^2_{x,y}}.
\end{equation}
\item For all $t\neq s$ and all $g\in L^{1}(\mathbb{R}^2)$
\begin{equation}
\label{cond2U:L2}
\|U(s)(U(t))^{*}g\|_{L^\infty_{x,y}} \lesssim |t-s|^{-\sigma} A \|g\|_{L^1_{x,y}}.
\end{equation} 
where $\sigma, A>0$. 
\end{itemize}
Then the estimates
$$\|U(t)f\|_{L^q_t L^r_{x,y}} \lesssim A^{\frac{1}{r'}-\frac{1}{2}} \|f\|_{L^2_{x,y}},$$
and
$$\left\|\int_{t<s} U(t)(U(s))^{*} F(s) ds\right\|_{L^q_t L^r_{x,y}} \lesssim A^{\frac{1}{r'}+\frac{1}{\tilde{r}'}-1} \|F\|_{L^{\tilde{q}'}_t L^{\tilde{r}'}_{x,y}},$$
hold for all sharp $\sigma$-admissible exponent pairs $(q,r)$ and $(\tilde{q}, \tilde{r})$ i.e. they satisfy
$$ \frac{1}{q} + \frac{\sigma}{r} = \frac{\sigma}{2}.$$
\end{lemma}
As a consequence, we have the following Strichartz estimates.
\begin{proposition}
\label{strichartz}
Let $$S(t) := \exp\left(it(\partial_{xx}-|D_y|)\right).$$
Then, for $f\in L^2_xH^s_y(\mathbb{R}^{2})$ with $s>\frac{1}{2}$, we have
\begin{equation}
\label{Strich:lin}
\|S(t)f\|_{L^4_TL^\infty_{x,y}} \lesssim \|f\|_{L^2_xH^{s}_y}
\end{equation}
and for $F \in L^1_T L^2_{x}H^{s}_y$
\begin{equation}
\label{Strich:non-homo}
\left\|\int_0^t S(t-s)F(s)ds\right\|_{L^4_TL^\infty_{x,y}} \lesssim \|F\|_{L^1_T L^2_{x}H^{s}_y} .
\end{equation}
\end{proposition}
\begin{remark}
The proofs of \eqref{Strich:lin} and \eqref{Strich:non-homo} does not seem to follow from 1D-Strichartz estimate because of a problem of changing $x$ and $y$ spatial norms. Instead, we adapt a standard proof using a semi-classical approach (for more details, see for example \cite{Bou-Tz}, \cite{Keel-Tao}, \cite{Robert}, \cite{Zworski} and the references therein).
\end{remark}
\begin{proof}
Let $\varphi\in C_{0}^{\infty}(\mathbb{R})$ 
be a cut-off function $0 \leq \varphi \leq 1$ and $h \in (0,1]$. Denote by 
$$S_h(t):=  S(t) \varphi(h |D_y|),$$
where the operator $\varphi(h |D_y|)$ is defined by 
$$\boF_y(\varphi(h |D_y|) f )(\xi) = \varphi(h |\xi|) \boF_y( f )(\xi).$$
We recall the Bernstein estimate (see, for example, 
\cite[Lemma 11.4]{K T V}) 
\[
\|e^{i(t-s)|D_{y}|}\varphi(h |D_y|)g_{1}\|_{L^p_y} 
\lesssim h^{\frac{1}{p} -\frac{1}{q}} \|e^{i(t-s)|D_{y}|}\varphi(h |D_y|) g_{1}\|_{L^q_y}, \quad 
\mbox{for any $p,q \geq 1$},
\]
and the dispersion estimate for $1$D Schr\"{o}dinger operator
(see e.g. Cazenave~\cite[Proposition 2.3]{Cazenave}),
\[
\|e^{i(t-s)\partial_{xx}}g_2\|_{L^\infty_x} \lesssim |t-s|^{-\frac{1}{2}} \|g_2\|_{L^1_x}, \quad 
\mbox{for any $g_{2} \in L^1_{x}(\mathbb{R})$}.
\] 
Then, from the two previous estimates 
with 
$g_{1} = e^{i (t - s)\partial_{xx}} f$ and $g_{2} = f$
and $p=\infty$ and $q=1$, we have
\begin{align*}
\|S_h(s)(S_h(t))^{*}f\|_{L^\infty_{x,y}} \lesssim h^{-1} \|e^{i(t-s)\partial_{xx}}\varphi(h |D_y|)f\|_{L^1_y L^\infty_{x}}\lesssim |t-s|^{-\frac{1}{2}} h^{-1} \|  f\|_{L^1_{x,y}}. 
\end{align*}
This allows us to apply Lemma \ref{Lem:Keel-Tao} because 
$S_{h}$ is a unitary operator on $L^{2}(\mathbb{R}^{2})$. 
Since $(4,\infty)$ and $(\infty,2)$
are sharp $\frac{1}{2}$-admissible exponent pairs, we obtain
$$\|S_h(t) f\|_{L^4_TL^\infty_{x,y}} 
\lesssim h^{-\frac{1}{2}} \| f\|_{L^2_{x,y}},$$
and 
$$ \left\|\int_{t<s} S_h(t)(S_h(s))^{*} F(s) ds\right\|_{L^4_TL^\infty_{x,y}} \lesssim h^{-\frac{1}{2}} \| F\|_{L^{1}_TL^{2}_{x,y}}.$$
By choosing $\tilde{\varphi} \in C^\infty_0(\R\setminus \{0\})$ satisfying $\tilde{\varphi}=1$ near the support of $\varphi$, we write
$$
S_h(t) f = S(t) \tilde{\varphi} (h |D_y|) \varphi(h |D_y|) f.
$$
We apply the two previous estimates with 
$S(t) \tilde{\varphi} (h |D_y|)$ and $\varphi(h |D_y|) f$ 
instead of $S_h(t)$ and $f$, respectively. 
Then, we obtain
\begin{equation}
\label{Strich:lin-1}
\|S_h(t) f\|_{L^4_TL^\infty_{x,y}} 
\lesssim h^{-\frac{1}{2}} \|\varphi(h |D_y|) f\|_{L^2_{x,y}},
\end{equation}
and 
\begin{equation}
\label{homo-strich-L6-1}
\left\|\int_{t<s} S_h(t)(S_h(s))^{*} F(s) ds\right\|_{L^4_TL^\infty_{x,y}} \lesssim h^{-\frac{1}{2}} \|\varphi(h |D_y|) F\|_{L^{1}_TL^{2}_{x,y}}.
\end{equation}
Now, we take $h=N^{-1}$ with $N$ is a dyadic number i.e. $N=2^n$, 
$n \in \N$. Note that
$$
S(t)f = \sum_{N = 2^{n}, n \in \mathbb{N}} S(t) \varphi(N^{-1} |D_y|) f, 
$$
This together with \eqref{Strich:lin-1} and 
the Cauchy Schwarz inequality for the sum on 
$N$ yields that 
\begin{align*}
\|S(t)f\|_{L^4_TL^\infty_{x,y}}  
&\lesssim 
\sum_{N = 2^{n}, n \in \mathbb{N}} \|S_{N^{-1}}(t)f\|_{L^4_T L^\infty_{x,y}} 
\lesssim \sum_{N = 2^{n}, n \in \mathbb{N}}  
N^{\frac{1}{2}} 
\|\varphi(N^{-1} |D_y|) f\|_{L^2_{x,y}} \\
& \leq 
\left(\sum_{N = 2^{n}, n \in \mathbb{N}} N^{2s} \|\varphi(N^{-1} |D_y|)
f\|_{L^2_{x,y}}^2\right)^\frac{1}{2} 
\left( \sum_{N = 2^{n}, n \in \mathbb{N}} N^{-2s + 1} \right)^{\frac{1}{2}}\\
&\lesssim \|f\|_{L^2_xH^{s}_y}, 
\end{align*}
where we have used the fact $s > 1/2$ in the last inequality.  
This leads to \eqref{Strich:lin}. 
Applying a similar argument to the above estimate, 
we infer from \eqref{homo-strich-L6-1} that
\begin{align*}
\left\|\int_{t<s} S(t)(S(s))^{*} F(s) ds\right\|_{L^4_TL^\infty_{x,y}}  &\lesssim \sum_N \left\|\int_{t<s} S_{N^{-1}}(t)(S_{N^{-1}}(s))^{*} F(s) ds\right\|_{L^4_TL^\infty_{x,y}}  \\
&\lesssim \sum_N  N^{\frac{1}{2}} \|\varphi(N^{-1} |D_y|)F\|_{L^{1}_TL^{2}_{x,y}}  \\
&\lesssim \int \left(\sum_N  N^{s} \|\varphi(N^{-1} |D_y|) F(t)\|_{L^2_{x,y}}^2\right)^\frac{1}{2} dt\\
& \lesssim \|F\|_{L^{1}_TL^2_xH^{s}_y}
\end{align*}
Thus, we see that \eqref{Strich:non-homo} holds.
\end{proof}
Next, we recall the following fundamental result (see for example Theorem 1.2.5 in \cite{Cazenave}).
\begin{lemma}
\label{thm1.2.5}
Consider two Banach spaces $X \hookrightarrow Y$, 
an open interval $I \subset \R$ and $1< p,q\leq \infty.$ 
Let $(u_n)_{n\in \N}$ be a bounded sequence in $L^q(I,Y)$ and 
$u: I \to Y$ be such that $u_n(t) \rightharpoonup u(t)$ in $Y$ 
as $n\to \infty$, for a.a. $t\in I$. 
If $(u_n)_{n\in \N}$ is bounded in $L^p(I,X)$ and if $X$ is reflexive, 
$u\in L^p(I,X)$ and $\|u\|_{L^p(I,X)} \leq \liminf_{n\to \infty} \|u_n\|_{L^p(I,X)} $.
\end{lemma}

Using Lemmas \ref{product-chain0}, 
\ref{product-chain}, \ref{thm1.2.5} and Proposition \ref{strichartz}, 
we shall show Theorem \ref{Cauchy}. 
\begin{proof}[Proof of Theorem \ref{Cauchy}]
The equation \eqref{WS} with $u|_{t=0} = u_{0}$ 
is equivalent to the following 
integral equation: 
\[
\psi(t) = S(t) \psi_{0} + i \int_{0}^{t} S(t - \tau)|\psi|^{p-1}\psi(\tau) d\tau.  
\]
We put 
\[
Y_{T} := L^{4}((-T, T); L_{x, y}^{\infty}(\mathbb{R}^{2})) 
\cap L^{\infty} ((-T, T); L^{2}_{x} H_{y}^{s}(\mathbb{R}^{2})),  
\]
where $s > 1/2$. 
Fix $T, M >0$, to be chosen later, and consider the set 
\[
E := \left\{u \in Y_{T} \mid 
\|u\|_{L_{T}^{4}((-T, T); L^{\infty}_{x, y}(\mathbb{R}^{2}))} \leq M, \; 
\|u\|_{L_{T}^{\infty}((-T, T); L^{2}_{x}H^{s}_{y}(\mathbb{R}^{2}))} \leq M
\right\}
\]
equipped with the distance 
\[
d(u, v) = \|u - v\|_{L^{4}_{T}((-T, T), L^{\infty}_{x, y}(\mathbb{R}^{2}))} 
+ \|u - v\|_{L^{\infty}_{T}((-T, T), L^{2}_{x,y}(\mathbb{R}^{2}))}. 
\]
We will split the proof into two steps. In the first one we will show that $(E,d)$ is a complete metric space. The second one is dedicated to the contraction mapping argument.

\underline{\textit{Step 1}:} 
We shall prove that $(E,d)$ is a complete metric space. It is sufficient to show that $E$ is a closed subspace in $L^{\infty}((-T, T), L^{2}_{x,y}(\mathbb{R}^{2})) \cap L^{4}((-T, T); L_{x, y}^{\infty}(\mathbb{R}^{2}))$ which is a complete metric space 
equipped with the distance $d$. 

Let $(u_n)_{n\in \N}$ be a sequence included into $E$ and satisfying $d(u_n,u) \to 0$ as $n\to \infty$ for some $u \in L^{\infty}((-T, T), L^{2}_{x,y}(\mathbb{R}^{2})) \cap L^{4}((-T, T); L_{x, y}^{\infty}(\mathbb{R}^{2}))$. Our aim is to prove that $u \in E$. We know that 
$$\|u\|_{L_{T}^{4}L^{\infty}_{x, y}} \leq M$$
is a direct consequence of the convergence $u_n \to u$ in $L^{4}((-T, T); L_{x, y}^{\infty}(\mathbb{R}^{2}))$ as $n\to \infty$. It remains to show that 
\begin{equation}
\label{u-in-Hs}
u \in  L^{\infty}((-T, T); L^{2}_{x}H^{s}_{y}(\mathbb{R}^{2})) \qquad \|u\|_{L_{T}^{\infty}L^{2}_{x}H^{s}_{y}} \leq M.
\end{equation}  
From the fact $\lim_{n \to \infty}d(u_{n}, u) = 0$, 
we can easily see that 
$u_n(t) \to u(t)$ in $L^{2}_{x,y}(\mathbb{R}^{2})$ 
for $\mbox{a.a.}\ t \in (-T, T)$. We apply Lemma \ref{thm1.2.5} for $X=L^2_xH^s_y(\R^2)$, $Y=L^{2}_{x,y}(\mathbb{R}^{2})$ and $p=q=\infty$ in order to obtain
$u(t) \in L^{2}_{x} H_{y}^{s}(\mathbb{R}^{2})$  and 
$$\|u(t)\|_{L^{2}_{x} H_{y}^{s}} \leq \liminf_{n\to \infty}  \|u_n(t)\|_{L^{2}_{x} H_{y}^{s}} \leq M,$$ for a.a. $t \in (-T, T)$. This proves \eqref{u-in-Hs}. Thus $u \in E$ and $(E,d)$ is a complete metric space.

\underline{\textit{Step 2}:} For each $u_{0} \in L^2_x H^{s}_{ y}(\mathbb{R}^{2})$, 
we put
\[
\mathcal{T}[u](t) := S(t)u_{0} + i \int_{0}^{t} S(t-\tau)|u|^{p-1}u(\tau) d\tau. 
\]
Thus, it is enough to show that $\mathcal{T}[u]$ is a contraction mapping 
on $E$.
First, we claim that $\mathcal{T}$ maps from $E$ to itself. 
By the Sobolev embedding $H_{y}^{s}(\mathbb{R}) \hookrightarrow 
L_{y}^{\infty}(\mathbb{R})\; (s>1/2)$, the Strichartz estimate 
\eqref{Strich:lin} and \eqref{Strich:non-homo}, 
the inequality \eqref{eq-chain} 
and the H\"{o}lder inequality, we have 
\begin{equation} \label{App-eq-0}
\begin{split}
 \quad \|\mathcal{T}[u](t)\|_{L_{T}^{4}L_{x, y}^{\infty}}
& \lesssim \|u_{0}\|_{L_{x}^{2}H^{s}_{y}} + \| |u|^{p-1}u
\|_{L^1_T L_{x}^{2}H^{s}_{y}} \\
& \lesssim \|u_{0}\|_{L_{x}^{2}H^{s}_{y}} + \| \|u\|_{L_{y}^{\infty}(\mathbb{R})}^{p-1}\|u\|_{H_{y}^{s}} \|_{L_{T}^{1} L_{x}^{2}} \\
& \lesssim 
\|u_{0}\|_{L_{x}^{2}H^{s}_{y}} + \| \|u\|_{L_{x, y}^{\infty}}^{p-1}
\|u\|_{L^{2}_{x}H_{y}^{s}}\|_{L^{1}_{T}} \\
& \lesssim \|u_{0}\|_{L_{x}^{2}H^{s}_{y}} + T^{\theta} \|u\|_{L_{T}^{4} L_{x, y}^{\infty}}^{p-1} 
\|u\|_{L^{\infty}_{T}L_{x}^{2}H_{y}^{s}},  
\end{split}
\end{equation}
where $\theta = (5-p)/4$. 
This yields that there exist $C_{1, 1}>0$
and $C_{1, 2} >0$ such that 
\begin{equation} \label{App-eq-1}
\|\mathcal{T}[u](t)\|_{L_{T}^{4} L_{x, y}^{\infty}}
\leq C_{1, 1} \|u_{0}\|_{L_{x}^{2}H_{y}^{s}} + C_{1, 2} T^{\theta} M^{p}. 
\end{equation}
Similarly, by the Strichartz estimate 
\eqref{Strich:non-homo}, 
the inequality \eqref{eq-chain}  
and the H\"{o}lder inequality, we obtain 
\begin{equation}
\label{App-eq-2.1}
\begin{split}
 \|\mathcal{T}[u](t)\|_{L_{T}^{\infty}
L_{x}^{2}H_{y}^{s}}
& \lesssim 
\|u_{0}\|_{L_{x}^{2}H_{y}^{s}} 
+ \| \||u|^{p-1}u(\tau)\|_{H_{y}^{s}}
\|_{L_{T}^{1}L_{x}^{2}} \\
& \lesssim 
\|u_{0}\|_{L_{x}^{2}H_{y}^{s}} + T^{\theta} \|u\|_{L_{T}^{4} L_{x, y}^{\infty}}^{p-1} 
\|u\|_{L^{\infty}_{T} L_{x}^{2}H_{y}^{s}}.
\end{split}
\end{equation}
If $1<p<5$, we write 
\begin{equation} \label{App-eq-2}
\|\mathcal{T}[u](t)\|_{L_{T}^{\infty}L_{x}^{2}
H_{y}^{s}}
\leq C_{2, 1} \|u_{0}\|_{L_{x}^{2}H_{y}^{s}} 
+ C_{2, 2} T^{\theta} M^{p}, 
\end{equation}
where $C_{2, 1}>0$ and $C_{2, 2} >0$. We take $M>0$ sufficiently large so that 
\[
\max\{C_{1, 1}, C_{2, 1}\}\|u_{0}\|_{L_{x}^{2}H_{y}^{s}(\mathbb{R}^{2})} 
< \frac{M}{2}. 
\]
For such $M>0$, we take $T >0$ sufficiently small so that 
$\max\{C_{2, 1}, C_{2, 2}\} T^{\theta}M^{p} < M/2$. 
This together with \eqref{App-eq-1} and \eqref{App-eq-2} yields that 
\[
\|\mathcal{T}[u](t)\|_{L_{T}^{4}L_{x, y}^{\infty}} \leq M
\quad {\rm and} \quad  \|\mathcal{T}[u](t)\|_{L_{T}^{\infty} L_{x}^{2}H_{y}^{s}}
\leq M.
\]
This implies that $\mathcal{T}[u] \in E$ when $1<p<5$. 

If $p=5$, $\theta=0$ and \eqref{App-eq-2.1} becomes 
\begin{equation} \label{App-eq-2-2}
\|\mathcal{T}[u](t)\|_{L_{T}^{\infty}L_{x}^{2}
H_{y}^{s}} \leq \frac{M}{2} + C_{2, 2}  M \|u\|_{L_{T}^{4} L_{x, y}^{\infty}}^{4}. 
\end{equation}
For $T$ small enough, we claim that
\begin{equation}
\label{small-L4}
\|u\|_{L_{T}^4L_{x, y}^{\infty}} \leq  \delta, 
\end{equation}
where $\delta>0$ is a small real number. Including \eqref{small-L4} into \eqref{App-eq-2-2}, we infer that 
$$\|\mathcal{T}[u](t)\|_{L_{T}^{\infty}L_{x}^{2}H_{y}^{s}} \leq \frac{M}{2} + C_{2, 2}  M \delta^{4}.$$
Hence, for $\delta$ small enough, we obtain
\[
\|\mathcal{T}[u](t)\|_{L_{T}^{4}L_{x, y}^{\infty}} \leq M
\quad {\rm and} \quad  \|\mathcal{T}[u](t)\|_{L_{T}^{\infty} L_{x}^{2}H_{y}^{s}}\leq M,
\]
and then $\mathcal{T}[u] \in E$ when $p=5$. 

We will prove now \eqref{small-L4}. We write the Duhamel formula 
$$u(t)= S(t)u_{0} + i \int_{0}^{t} S(t-\tau)|u|^{4}u(\tau) d\tau.$$
Using \eqref{Strich:non-homo} and H\"{o}lder estimate, we get
\begin{align}
\|u\|_{L_{T}^4 L_{x, y}^{\infty}}&\lesssim 
\|S(t)u_{0}\|_{L_{T}^4 L_{x, y}^{\infty}} 
+ \||u|^{4} u\|_{L_{T}^1 L_{x}^{2}H^{s}_{y}} \nonumber\\
&\lesssim \|S(t)u_{0}\|_{L_{T}^4L_{x, y}^{\infty}} + \|u\|_{L_{T}^{4}L_{x, y}^{\infty}}^4 \|u\|_{L^{\infty}_{T} L_{x}^{2}H_{y}^{s}}\nonumber\\
& \leq \|S(t)u_{0}\|_{L_{T}^4L_{x, y}^{\infty}} + C_{2,2} M \|u\|_{L_{T}^{4}L_{x, y}^{\infty}}^4 .
\label{small-L4-bound}
\end{align}
By \eqref{Strich:lin} and the dominated convergence, we have
$$\|S(t)u_{0}\|_{L_{T}^4L_{x, y}^{\infty}} \to 0 \quad {\rm as} \quad T\to 0. $$
So that, for sufficiently small $T>0$,
\begin{equation}
\label{small-L4-lin}
\|S(t)u_{0}\|_{L_{T}^4L_{x, y}^{\infty}} \leq \frac{\delta}{2}. 
\end{equation}
Assume by contradiction that \eqref{small-L4} does not hold i.e. $\|u\|_{L_{T}^4L_{x, y}^{\infty}} = q \delta$ with $q >1$. From \eqref{small-L4-bound} and \eqref{small-L4-lin}, we obtain
$$q \delta \leq \frac{\delta}{2} + C_{2,2} M q^4 \delta^4 < \delta,$$
for $\delta>0$ small enough. Thus $q<1$ which is a contradiction. This finishes the proof of \eqref{small-L4}.

Next, we shall show that $\mathcal{T}$ is a contraction 
mapping on $E$. 
Let $u, v \in E$. For $1<p<5$, as in the similar estimate to \eqref{App-eq-0} and \eqref{App-eq-2.1}, we have 
\begin{equation*}
\|\mathcal{T}[u] - \mathcal{T}[v]\|
_{L_{T}^{4}L_{x, y}^{\infty}} 
\lesssim T^{\theta} M^{p-1} d(u, v),
\end{equation*}
and
\begin{equation*}
\|\mathcal{T}[u] - \mathcal{T}[v]\|
_{L_{T}^{\infty}L_{x,y}^{2}} \lesssim T^{\theta} M^{p-1} d(u, v). 
\end{equation*}

Taking $T > 0$ sufficiently small, we have $d(\mathcal{T}[u], \mathcal{T}[v]) < \frac{1}{2} d(u, v)$. 
Therefore, we deduce that $\mathcal{T}$ is a contraction mapping.

If $p=5$, we have in a similar way to \eqref{App-eq-0} and \eqref{App-eq-2.1}, we obtain
\begin{equation*}
\|\mathcal{T}[u] - \mathcal{T}[v]\|
_{L_{T}^{4} L_{x, y}^{\infty}} 
\lesssim \|u\|_{L_{T}^4L_{x, y}^{\infty}}^4 d(u, v),
\end{equation*}
and
\begin{equation*}
\|\mathcal{T}[u] - \mathcal{T}[v]\|
_{L_{T}^{\infty}L_{x,y}^{2}} \lesssim \|u\|_{L_{T}^4L_{x, y}^{\infty}}^4 d(u, v). 
\end{equation*}
For $T > 0$ sufficiently small, we infer from \eqref{small-L4} that $d(\mathcal{T}[u], \mathcal{T}[v]) < \frac{1}{2} d(u, v)$. 
Hence, $\mathcal{T}$ is a contraction mapping. 

This completes the proof. 
\end{proof}

\begin{remark}
\label{Rk5.3}
\eqref{small-L4} is also true for $1<p<5$. However, we use the smallness in $T$ instead to obtain dependence of the maximal time of existence with respect to the $L^2_xH^s_y$ norm of the initial data. Namely, we see that 
\[
\lim_{ t \to T_{\max} } \| u(t)\|_{L^2_xH^s_y} = \infty. 
\]
On the other hand, the method used in the case $p=5$ gives a different blow-up criteria depending on the Strichartz norm (see Chapter 4 in Cazenave \cite{Cazenave} for more details). 
\end{remark}

\appendix
\section{Some regularity result} \label{app-regularity}
In this appendix, we will show that the function $\psi_{\omega}$ introduced in Lemma \ref{thm3-1-1} and defined by \eqref{eq3-2} belongs to the energy space $X$. We recall that this fact is important to prove Lemma \ref{thm3-2}. 

We have the following result. 
\begin{proposition} \label{regularity}
Assume that $7/3 < p < 5$. 
Let $\psi_{\omega}$ be a function defined by \eqref{eq3-2}. 
Then, we have $\psi_{\omega} \in X$. 
\end{proposition}
By the scaling property \eqref{scale-Q}, 
we may assume that $\omega = 1$. 
We put 
\[
R_{1} = \partial_{\omega} Q_{\omega} |_{\omega = 1} 
= \frac{1}{p-1} Q_{1} + \frac{x \partial_{x} Q_{1}}{2} 
+ y \partial_{y} Q_{1}. 
\]
Then, $R_{1}$ satisfies the following: 
\begin{equation} \label{linear-Q}
- \partial_{xx} R_{1} + |D_{y}| R_{1} + R_{1} - p Q_{1}^{p-1} R_{1} = - Q_{1}. 
\end{equation}
Here, we recall the results of \cite{Esfahani-Pastor2}.

\begin{theorem}[Theorem 4.7 of \cite{Esfahani-Pastor2}]
\label{regular-1}
Let $\omega > 0$ and $1 < p < 5$. 
For any $k \in \mathbb{N}$ and any solution $Q_{\omega}$ 
to \eqref{sp}, we have $Q_{\omega} \in H^{k} 
\cap C^{\infty}(\mathbb{R}^{2})$.  
Moreover, all its derivatives are bounded and tends to zero 
at infinity. 
\end{theorem}

\begin{theorem}[Theorem 5.9 of \cite{Esfahani-Pastor2}]
\label{regular-2}
Let $\omega > 0$ and $1 < p < 5$. 
There exists $\sigma_{0} > 0$ such that 
for any $0 \leq \sigma < \sigma_{0}$ and $0 \leq s < 3/2$, 
we have 
\[
(e^{\sigma |x|} + |y|^{s}) Q_{\omega}(x, y) \in L^{1} \cap L^{\infty}
(\mathbb{R}^{2}).  
\]
\end{theorem}
It follows from Theorems \ref{regular-1} and \ref{regular-2} that 
$Q_{\omega}^{p-1} x \partial_{x} Q_{\omega}, 
Q_{\omega}^{p-1} y \partial_{y} Q_{\omega} \in L^{2}(\mathbb{R}^{2})$ 
for $p > 7/3$. 
Thus, if we put 
$P(x, y) : = - Q_{1} + p Q_{1}^{p-1} R_{1}$,  
we see that $P \in L^{2}(\mathbb{R}^{2})$ 
To prove Theorem \ref{regularity}, 
we shall also apply the following result of Lizorkin~\cite{Li}
\begin{proposition}[Lizorkin~\cite{Li}] \label{HM}
Let $\Phi: \mathbb{R}^{d} \to \mathbb{R}$ be $C^{d}$ 
for $|\xi_{i}| > 0, j = 1, 2, \ldots, d$. 
Assume that there exists $M >0$ such that 
\begin{equation} \label{ap}
\biggl|\xi_{1}^{k_{1}} \cdots \xi_{d}^{k_{d}}
\frac{\partial^{k} \Phi}{\partial \xi_{1}^{k_{1}} \cdots \partial \xi_{d}^{k_{d}}} \biggl|
\leq M
\end{equation}
with $k_{i} = 0$ or $1$ ($i = 1, 2, \ldots, d$), 
$k = k_{1} + k_{2} \cdots + k_{d} = 0, 1, 2, \ldots, d$. 
Then, we have $\Phi \in M_{q}(\mathbb{R}^{d}) \; (1 < q < \infty)$. 
Namely, there exists $C>0$ such that 
\[
\|\mathcal{F}^{-1}\left[\Phi \mathcal{F}[f] \right]\|_{L^{q}} 
\leq C \|f\|_{L^{q}}
\]
for $f \in \mathcal{S}(\mathbb{R}^{d})$. 
\end{proposition}
We are now in position to prove Theorem \ref{regularity}. 
\begin{proof}[Proof of Theorem \ref{regularity}]
Let 
\[
\Phi_{1}(\xi_{1}, \xi_{2}) = \frac{1}{|\xi_{1}|^{2} + |\xi_{2}| + \omega}, \quad 
\Phi_{2}(\xi_{1}, \xi_{2}) = \frac{- \xi_{1}^{2}}{|\xi_{1}|^{2} + |\xi_{2}| + \omega}, \]
\[ \Phi_{3}(\xi_{1}, \xi_{2}) = \frac{|\xi_{2}|}{|\xi_{1}|^{2} + |\xi_{2}| + \omega}. 
\]
We can easily check that $\Phi_{i}\; (i = 1, 2, 3)$ satisfies 
the condition \eqref{ap}.

It follows from \eqref{sp} that 
\begin{equation} \label{regular-1-1}
\begin{split}
& R_{1} = \mathcal{F}^{-1}\left[\Phi_{1}(\xi_{1}, \xi_{2}) 
\mathcal{F}(P)\right], \quad 
\partial_{xx} R_{1} = \mathcal{F}^{-1}\left[\Phi_{2}(\xi_{1}, \xi_{2}) 
\mathcal{F}(P)\right], \quad \\
& |D_{y}|R_{1} = \mathcal{F}^{-1}\left[\Phi_{3}(\xi_{1}, \xi_{2}) 
\mathcal{F}(P)\right]. 
\end{split}
\end{equation}
Since $P \in L^{2}(\mathbb{R}^{2})$, 
we have 
$R_{1}, \partial_{xx}R_{1}, |D_{y}|Q_{1} \in L^{2}
(\mathbb{R}^{2})$. 
This yields that $R_{1} \in X$, which completes the proof. 
\end{proof}

\section{Tools for compactness}\label{compactness}
\begin{lemma}[Brezis and Lieb \cite{Brezis-Lieb}]\label{16/03/25/01:03}
Assume that $1 < q < \infty$. 
Let $\{u_{n}\}$ be a bounded sequence in $L^{q}(\mathbb{R}^{d})$ such that 
\begin{equation}\label{16/03/25/11:11}
\lim_{n\to \infty}u_{n}(x)=u_{\infty}(x) 
\qquad 
\mbox{almost all $x\in \mathbb{R}^{d}$}
\end{equation}
for some function $u_{\infty} \in L^{q}(\mathbb{R}^{d})$. 
Then, we have 
\begin{equation}
\label{16/03/25/11:12}
\lim_{n\to \infty}
\int_{\mathbb{R}^{d}} 
\big| |u_{n}|^{q}-|u_{n}-u_{\infty}|^{q}-|u_{\infty}|^{q} \big| \,dx =0. 
\end{equation}
\end{lemma}

\begin{thank}
The authors are very grateful to Patrick Gerard for introducing
them to the problem and his interesting comments and suggestions. We thank also Takafumi Akahori for very helpful discussions leading to the improvement of the text.
This work was done while H.K. was visiting at University of Victoria. 
H.K. thanks all members of the Department of 
Mathematics and Statistics for their warm hospitality. The work of H.K. was supported by JSPS KAKENHI Grant Number JP17K14223. S.I was supported by NSERC grant (371637-2014). Y.B was supported by PIMS grant and NSERC grant (371637-2014).

\end{thank}

\bibliographystyle{plain}

\vspace{24pt}

\end{document}